\documentclass[makeidx]{amsart}
\usepackage[active]{srcltx}
\usepackage{amsmath}
\usepackage{color}
\usepackage{amssymb}
\usepackage{epsfig}
\usepackage{hyperref}
\newtheorem{Proposition}{Proposition}
  \newtheorem{Remark}{Remark}
  \newtheorem{Corollary}[Proposition]{Corollary}
  \newtheorem{Lemma}[Proposition]{Lemma}
  
  \newtheorem{Theorem}{Theorem}
 
 \newtheorem{Definition}[Proposition]{Definition}
 \newtheorem{Note}[Remark]{Note}

\def\L{\left}
\def\R{\right}

\def\blackslug{\hbox{\hskip 1pt \vrule width 4pt height 8pt depth 1.5pt
\hskip 1pt}}
\def\qed{\quad\blackslug\lower 8.5pt\null\par}
 
\def\CC{\mathbb{C}}
 \def\RR{\mathbb{R}}

\def\Re{\mathrm{Re}}

\def\ge{\geqslant}
\def\le{\leqslant}
\def\ve{\varepsilon}
\def\pa{\text{\small $\boldsymbol{\Pi}$}}
\def\epsilon{\varepsilon}
 
\def\CC{\mathbb{C}}
 \def\RR{\mathbb{R}}

\def\Re{\mathrm{Re}}

\def\ds{\displaystyle}
\def\P1{\text{P}_{\rm I}}
 \def\({\left(} \def\){\right)} 
\makeindex
\author{O. Costin$^1$, M. Huang$^2$, S. Tanveer$^1$} \address{$^1$ Mathematics Department\\The Ohio State University\\Columbus, OH 43210 } \address{$^2$ Mathematics Department\\The University of Chicago, IL 60637
}
\title{Proof of the Dubrovin conjecture and analysis of the tritronqu\'ee solutions of $P_I$}

\begin{document}
$ $ \vskip -0.2cm
\begin{abstract}
  We show that the tritronqu\'ee solution $y_t$ of the Painlev\'e equation
  $\P1$  that behaves algebraically  for large $z$ with $\arg z=\pi/5$, is analytic in a region containing the sector $\left \{ z\ne 0\\,
    \arg z \in \left [ -\frac{3\pi}{5}  , \pi \right ] \right \}$ and the disk
  $ \left \{ z: |z| < \frac{37}{20} \right \}$.  This implies the Dubrovin conjecture, an important open problem in the theory
  of Painlev\'e transcendents.  The method, building
  on a technique developed in \cite{NLS}, is general and constructive. As a
  byproduct, we obtain the value of the tritronqu\'ee and its derivative at
  zero, also important in applications, within less than $1/100$ rigorous error
  bounds.
\end{abstract}
\vskip -2cm
\maketitle
%\tableofcontents

\section{Introduction and Main result}
Understanding the global behavior in $\CC$ of the tritronqu\'ee solutions (see
below) of the Painlev\'e equation $\P1$  is essential
in a number of problems such as the  critical behavior in the NLS/Toda
lattices (\cite{Dubrovinc}, \cite{DubrovinJPA}) 
and the analysis of the cubic oscillator (\cite{Partial1}).   Considerations
related to the behavior of NLS/Toda solutions corroborated by numerical
evidence led to the conjecture in the Dubrovin-Grava-Klein paper \cite{Dubrovinc}  that the
tritronqu\'ee solutions are analytic in a neighborhood of the origin $O$ and in a
sector of width $8\pi/5$ containing $O$ (cf. also \cite{Claeys} and \cite{DubrovinJPA}). A number of partial
  results on this question have been obtained so far (see
  e.g. \cite{Nalini},\cite{survey},\cite{Partial2}--\cite{Partial1}) but, in spite of the
  existence of an underlying Riemann-Hilbert representation, at the time of the present paper the conjecture is
  still open. 

The purpose of the present paper is to prove  the Dubrovin conjecture alongside  other
  results about the tritronqu\'ees.

\centerline{*}

We write $\P1$ in the form
\begin{equation}
  \label{eq:p1}
  y''=6y^2+z
\end{equation}
See, e.g., \cite{KitaevKapaev} for an excellent review of classical results about this equation.

{\em Tritronqu\'ees}. There are exactly five solutions of \eqref{eq:p1} which are analytic for large
$z$ in a sector of width $8\pi/5$. These special solutions called {\em
  tritronqu\'ees} are obtained from each-other  through the five-fold symmetry
of $\P1$, $y\mapsto e^{4 i k \pi/5} y \left ( e^{2 \pi i k/5} z \right ),\,k=0,\ldots,4$
(cf. \cite{Kapaev-2004},\cite{ KitaevKapaev}, \cite{Nalini}). 
To understand their properties it is
clearly enough to analyze one of them;  we choose  the solution $y_t$ 
uniquely defined
by the property $y(z)\sim i\sqrt{z/6}[1+O(z^{-5/2})]$ as $z\to +\infty$
(cf. {\it{e.g}}.  Proposition 2 and Theorem 3 in \cite{Nalini}).
We will also use a standard normalization of $\P1$ \cite{invent}, similar to the Boutroux
form. After the change of variables
\begin{equation}
\label{0.1}
x = \frac{e^{i \pi/4}}{30}  \left ( 24 z \right )^{5/4} \\,
~y(z) = i \sqrt{\frac{z}{6}} \left ( 1 - \frac{4}{25 x^2} + h(x) \right ) 
\end{equation} 
$\P1$ becomes
\begin{equation}
\label{1}
h^{\prime \prime} + \frac{1}{x} h^\prime  - h = \frac{h^2}{2} + \frac{392}{625 x^4}. 
\end{equation}
The tritronquu\'ee $y_t$ corresponds through \eqref{0.1} to the unique solution of \eqref{1} with the behavior  $h (x) = O \left ( x^{-4} \right )$ as $x\to +i
\infty$ (cf. \cite{Nalini}).
The main result of this paper is
  \begin{Theorem}
\label{Thm1}
The tritronqu\'ee $y_t$ is analytic in the region
\begin{equation}
  \label{eq:mainreg}
 \mathcal{D}:= \Big\{z\ne 0:\arg\,z \in[-3\pi/5,\pi]\Big\}\cup \Big\{z:|z|<\tfrac{37}{20}\Big\}
\end{equation}
\end{Theorem}
  \begin{Corollary}
   Dubrovin's conjecture, stating that $y_t$ is free of singularities in $\{z:\arg z \in
\left ( -\frac{3}{5} \pi, \pi \right )\} $,  holds.
  \end{Corollary}

\subsection{Strategy of the proof}
A reflection symmetry w.r.t. the bisector $\arg z=\pi/5$
allows us to restrict the analysis to $\mathcal{D}_1$, which is the half of 
$\mathcal{D}$ below its bisector. Let
$\mathcal{N}(y):=y''-6y^2-z$. We first define a quasi-solution $y_0$ in $\mathcal{D}_1\cap \{z:|z| \ge r_0 \}$ ($r_0$ is roughly $1.52$)
calculated from the large $z$ behavior of $y$ which is
matched for smaller $z$ to  a polynomial $P$ used on the line segment  $|z|\in [0,1.7],
\,\,\arg z =\pi/5$. This $y_0$ is a quasi-solution in the sense that
$\mathcal{N}(y_0)$ is small in $L^\infty$ and $\|y_t-y_0\|$
is subsequently shown to be small. Then the equation satisfied by $E =y_t-y_0$
is weakly nonlinear,
\begin{equation}
  \label{eq:eqnonl}
  L E =-\mathcal{N}(y_0)+\mathcal{N}_1(E) 
\end{equation}
where $L=\frac{\partial \mathcal{N} }{\partial y}|_{y=y_0}$ and $\mathcal{N}_1(E) =O(\|E \|^2)$. 
To determine the properties of $E $ we rewrite as usual \eqref{eq:eqnonl} in integral form, 
\begin{equation}
  \label{eq:eqnonli}
  E =-L^{-1}\mathcal{N}(y_0)+L^{-1}\mathcal{N}_1(E) 
\end{equation}
where the constants in \eqref{eq:eqnonli} are chosen based on the
asymptotic behavior of $y_t$. The solution of \eqref{eq:eqnonli},  its
size and properties are found by using standard contractive mapping
arguments. The domain   $\mathcal{D}_1 \cap \left \{ z: |x(z)| \ge 3 \right \}$ (see \eqref{0.1} )  is
further subdivided into  $\arg z\in [0,\tfrac{\pi}{5}]$,  $\arg z\in [-\frac{2\pi}{5},0]$ 
and $\arg z\in [-\frac{3\pi}{5},-\frac{2\pi}{5}]$; in $x$ these are, roughly,  
$\Omega_{1}$, $\Omega_2$ and $\Omega_3$  resp., see Fig. \ref{Fig1}. In the first two regions, where slightly different technical arguments are used, $y_0$ 
is simply obtained by setting $h=0$ in \eqref{0.1}.
 Finally, in $\Omega_3$, which is ``close to the pole region'' (because $|z|$ is as small as 1.52, and beyond $\arg z=-\tfrac{3\pi}{5}$, $y_t$ {\em does have} poles), 
we need to use  a two-scale  expansion  in powers of $1/x(z)$ and $\exp(-x(z))$, 
which we derive from the results in \cite{invent}.

On the ray $\arg z= \pi/5$, which corresponds to 
the positive imaginary $x$ axis, $\| E \|_{\infty}, \| E^\prime \|_\infty \le 6.5 \cdot 10^{-3}$ 
and the polynomial $P$ mentioned above is obtained by projecting the Taylor series of $y_0$ at $z=1.7 e^{i \pi/5}$ on Chebyshev polynomials   
(for economy of representation, since Chebyshev polynomials are known as being nearly optimal in terms of precision of representation for a given degree). From the polynomial we calculate the values of $y_t(0)$ and $y'_t(0)$ with errors less than $10^{-2}$. We then show, by straightforward estimates of Taylor coefficients that any solution with initial conditions in this range has radius of analyticity at least $37/20=1.85$. Combining all these results this proves the existence of an analytic solution in $\mathcal{D}_1$ 
with the same asymptotic behavior as $y_t$. The result follows from the known uniqueness of the solution with these properties.

As noted, the analysis of the integral equation uses standard
contractive mapping arguments. What makes this approach possible is,
among other factors, the high accuracy of the two-scale expansions in
\cite{invent} allowing for good control of the solution down to
$|z|=r_0$. This type of approach should work in a much wider
generality, and does not use the Riemann-Hilbert reformulation of
$\P1$. The only nonelementary part of the approach is in constructing
the quasi-solution $y_0$ (of course, from the point of view of the
proof, the ansatz $y_0$ could simply be defined without
justification).

  \begin{Note}
    \rm{
(i) As a byproduct we find within error bounds smaller than $1/100$ the values of $y(0)$ and $y'(0)$,  which are also needed in
applications,  and these values are
consistent with the  numerical calculation by Joshi and Kitaev  \cite{Nalini}.

(ii) The lower bound $1.85$ 
is not optimal, yet not far from the numerically
obtained radius of analyticity, $\sim 2.38$ (\cite{Nalini}). 
The methods we use can, 
in principle, be adapted to rigorously calculate the tritronqu\'ee and the position of its first pole with any prescribed accuracy.}
  \end{Note}

\section{The tritronqu\'ee $y_t$ on the antistokes line $z=e^{i \pi/5} \mathbb{R}^+$}
\begin{figure}
  \centering
%%%%%%%%%%%%%%%%%%%Figure1
\begin{picture}(0,0)%
\includegraphics{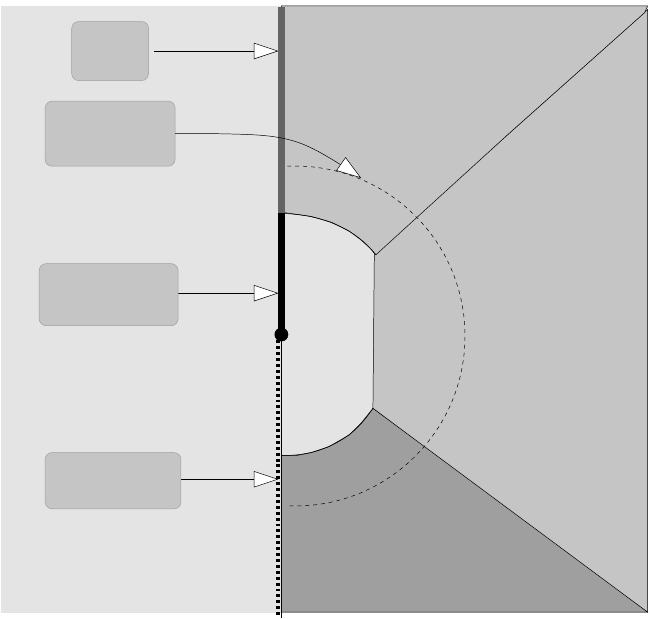}%
\end{picture}%
\setlength{\unitlength}{2072sp}%
\begingroup\makeatletter\ifx\SetFigFont\undefined%
\gdef\SetFigFont#1#2#3#4#5{%
  \reset@font\fontsize{#1}{#2pt}%
  \fontfamily{#3}\fontseries{#4}\fontshape{#5}%
  \selectfont}%
\fi\endgroup%
\begin{picture}(9879,9413)(1294,-8578)
\put(6121, 29){\makebox(0,0)[lb]{\smash{{\SetFigFont{9}{10.8}{\familydefault}{\mddefault}{\updefault}{\color[rgb]{0,0,0}$\Omega_1$}%
}}}}
\put(5986,-331){\makebox(0,0)[lb]{\smash{{\SetFigFont{9}{10.8}{\familydefault}{\mddefault}{\updefault}{\color[rgb]{0,0,0}$y_0=i\sqrt{\frac{z}6}(1-\tfrac{4}{25x^2})$}%
}}}}
\put(5761,-7891){\makebox(0,0)[lb]{\smash{{\SetFigFont{9}{10.8}{\familydefault}{\mddefault}{\updefault}{\color[rgb]{0,0,0}$\Omega_3$}%
}}}}
\put(5761,-8251){\makebox(0,0)[lb]{\smash{{\SetFigFont{9}{10.8}{\familydefault}{\mddefault}{\updefault}{\color[rgb]{0,0,0}$y_0=i\sqrt{\frac{z}6}(1-\tfrac{4}{25x^2}+h_0) \text{ [see \eqref{1.0}}]$}%
}}}}
\put(8686,-3706){\makebox(0,0)[lb]{\smash{{\SetFigFont{9}{10.8}{\familydefault}{\mddefault}{\updefault}{\color[rgb]{0,0,0}$\Omega_2$}%
}}}}
\put(7381,-4066){\makebox(0,0)[lb]{\smash{{\SetFigFont{9}{10.8}{\familydefault}{\mddefault}{\updefault}{\color[rgb]{0,0,0}$y_0=i\sqrt{\frac{z}6}(1-\tfrac{4}{25x^2})$}%
}}}}
\put(2071,-6541){\makebox(0,0)[lb]{\smash{{\SetFigFont{9}{10.8}{\familydefault}{\mddefault}{\updefault}{\color[rgb]{0,0,0}Cut in $x$ plane}%
}}}}
\put(2791,-16){\makebox(0,0)[lb]{\smash{{\SetFigFont{9}{10.8}{\familydefault}{\mddefault}{\updefault}{\color[rgb]{0,0,0}$\Omega_I$}%
}}}}
\put(2296,-1276){\makebox(0,0)[lb]{\smash{{\SetFigFont{9}{10.8}{\familydefault}{\mddefault}{\updefault}{\color[rgb]{0,0,0}$|z|=1.85$}%
}}}}
\put(2071,-3706){\makebox(0,0)[lb]{\smash{{\SetFigFont{9}{10.8}{\familydefault}{\mddefault}{\updefault}{\color[rgb]{0,0,0}$g_0,\ \text{see \eqref{eq:defgt}}$ }%
}}}}
\end{picture}%
  \caption{The domains of analysis and the corresponding quasi-solutions.}
% \label{fig:ro}
\label{Fig1}
\end{figure}
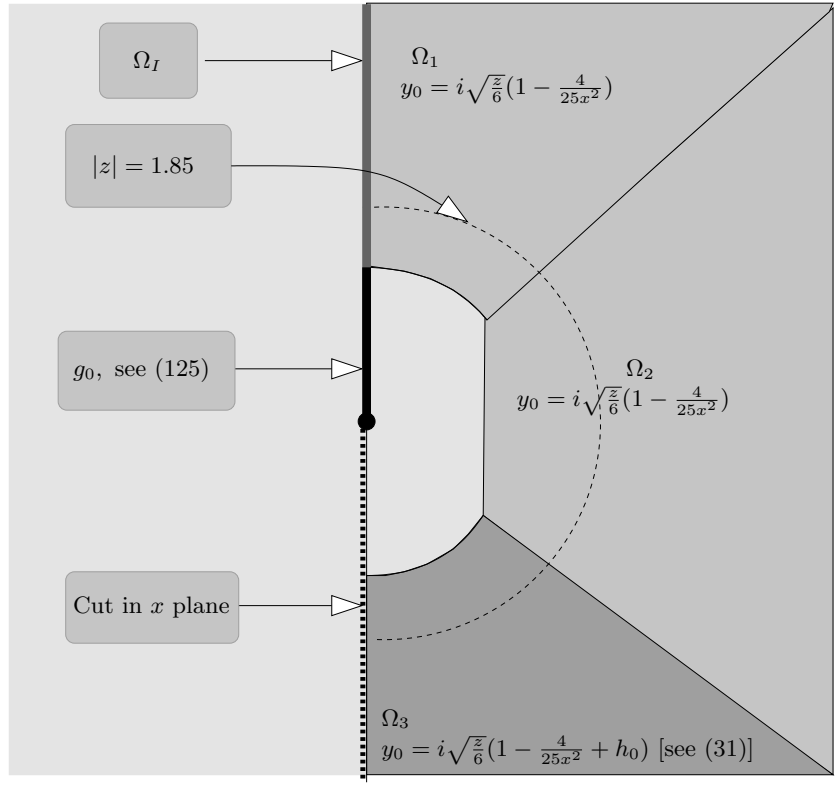
%%%%%%%%%%%%%%%%%%%%%%%%end Fig 1
First, we consider behavior of $y_t$ on the antistokes line $z=e^{i \pi/5} \mathbb{R}^+$ for
$|z| \ge \frac{1}{24} \left ( 30 \right )^{4/5} (= 0.633\cdots<1.7)$. As mentioned, we find 
$y_t (z_0)$ and $y_t^\prime (z_0)$ for
$z_0 = 1.7 e^{i \pi/5}$, with sufficiently sharp estimates. This is
used later
to show that $y_t$ and in fact any solution of $\P1$ with $y (z_0)$ and $y^\prime (z_0)$
within the range implied by these error bounds has
a power series at the origin with radius of convergence
exceeding $1.85$.  After the substitution
\begin{equation}
\label{1.0a}
h(x) =: \frac{H(x)}{\sqrt{x}}, 
\end{equation}
 (\ref{1}) implies 
\begin{equation}
\label{0.1.1}
H^{\prime \prime} - \left (1 - \frac{1}{4 x^2} \right ) H = 
\frac{H^2}{2 \sqrt{x}} + \frac{392}{625 x^{7/2}} 
\end{equation}
We define
\begin{equation}
\label{0.2}
\Omega_I :=
\left \{ x \in i \mathbb{R}^+, |x| \ge \rho \right \};\ \ \text{where
  $\rho\geqslant 1$}
\end{equation}
and   the Banach space 
\begin{equation}
  \label{eq:defSI}
\mathcal{S}_I=\{H\in C \left (\Omega_I, \CC \right ): \|H\|<\infty\}
\end{equation}
where $\|\cdot\|$ is the weighted norm
\begin{equation}
\label{0.3}
\| H \| = \sup_{x \in \Omega_I} \Big | x^{5/2} H \Big |
\end{equation}
We invert (\ref{0.1.1}) in  $\mathcal{S}_I$ to obtain the integral equation
\begin{equation}
\label{0.4}
H(x) = H_0 (x) +
\int_{i \infty}^x \sinh (x-t) 
\left \{ -\frac{H(t)}{4 t^2} + \frac{H^2 (t)}{2 \sqrt{t}} \right \} ~dt  
=:\mathcal{N} \left [ H \right ] (x), 
\end{equation}
where 
\begin{equation}
\label{0.5}
H_0 (x) = \int_{i \infty}^x \sinh (x-t) ~\frac{392}{625 t^{7/2}} dt  
\end{equation}
In the process of inverting (\ref{0.1.1}), no 
nonzero linear combination
$C_1 e^{x} +C_2 e^{-x} $ may be added to the right side of (\ref{0.4}) as
$e^{\pm x}\notin \mathcal{S}_I$.
We note that $|x|=\rho=1$ corresponds to  $|z|=\frac{30^{4/5}}{24}=0.6331\cdots $. 
\begin{Lemma}
\label{lem0}
There exists a unique solution $H$ to 
(\ref{0.4}) and therefore of (\ref{0.1.1}) 
in $\mathcal{S}_I$.
\end{Lemma}
\begin{proof}
Using $\left | \sinh (x-t) \right | \le 1$ for $x,t\in \Omega_I$
and  (\ref{0.5}) we obtain 
\begin{equation}
  \label{eq:eqH0e}
  |H_0 (x)| \le \frac{784}{3125 |x|^{5/2}}\Rightarrow \| H_0 \| \le \frac{784}{3125}
\end{equation}
Since $|H(t) |\le |t|^{-5/2} ~\| H \| $,
it follows that  
$$ \Big | \int_{i \infty}^x \sinh (x-t) 
\left \{ -\frac{H(t)}{4 t^2} ~+~\frac{H^2 (t)}{2 \sqrt{t}} \right \}~dt  \Big |
\le \frac{\|H\|}{14 |x|^{7/2}} + \frac{ \| H \|^2}{9 |x|^{9/2}}, $$   
Therefore, it follows that 
$$ \| \mathcal{N} \left [ H \right ] \| 
\le \| H_0 \| + \frac{ \|H \|}{14 \rho}  
+\frac{ \| H \|^2}{9 \rho^2}, $$   
and  similarly,
$$ \| \mathcal{N} \left [ H_1 \right ] - \mathcal{N} \left [ H_2 \right ] \| 
\le  \left ( \frac{1}{14 \rho} +\frac{ \| H_1 \| + \| H_2 \| }{9 \rho^2} \right )
\| H_1 - H_2 \| $$
Consider now the ball $B\subset \mathcal{S}_I$ of radius
$(1+\epsilon) \| H_0 \|$ for some $\epsilon > 0$ to be chosen shortly.
For any 
$H, H_1, H_2 \in B$ we obtain from the inequalities above
\begin{equation}
  \label{eq:e12}
   \| \mathcal{N} \left [ H \right ] \| \le 
\| H_0 \| \left ( 1 + \frac{1+\epsilon}{14 \rho} +   
\frac{(1+\epsilon)^2 \| H_0 \|}{9 \rho^{2} } \right ) 
\end{equation}
\begin{equation}
  \label{eq:e13}
  \| \mathcal{N} \left [ H_1 \right ] - \mathcal{N} \left [ H_2 \right ] \|
\le \left \{ \frac{1}{14 \rho} + \frac{2}{9 \rho^2}   (1+\epsilon) \| H_0 \|\right \} 
\|H_1 - H_2 \| 
\end{equation}
The estimates \eqref{eq:e12} and \eqref{eq:e13} imply 
$\mathcal{N}:B\to B$ and is contractive if
\begin{equation}
\label{0.6}
\frac{1+\epsilon}{14 \rho} +
\frac{(1+\epsilon)^2 \| H_0 \|}{9 \rho^{2} } \le \epsilon \\,\text{and} ~
\frac{1}{14 \rho} + \frac{2}{9 \rho^2}   (1+\epsilon) \| H_0 \|  < 1 
\end{equation}
Recalling that $\rho \ge 1$, \eqref{eq:eqH0e} implies that the conditions
  in (\ref{0.6}) are satisfied if $\epsilon = \frac{3}{20} $,.  Thus,
  (\ref{0.4}) has a unique solution $H$ in a ball of size $\frac{23}{20} \|H_0
  \| $-- and so does the equivalent equation \ref{0.1.1}). Reverting the
  changes of variables, the corresponding $y$ is $y_t$ since
  $H\in\mathcal{S}_I$ implies the decay of $y$ at $i \infty$ that
  uniquely determines $y_t$.
\end{proof}
\begin{Lemma}
\label{lem01}
With $H$ as in Lemma \ref{lem0}, 
we have for $x\in \Omega_I$,
$$\Big | H^\prime (x) \Big | \le 
\left ( \frac{1}{14 |x|^{7/2}} \| H \|+
\frac{1}{9 |x|^{9/2}} \| H \|^2
+ \frac{392}{625 |x|^{7/2}}  \right ),  
$$
\end{Lemma}
\begin{proof}
Differentiating \eqref{0.4} and using \eqref{0.5} we obtain
\begin{equation}
  \label{eq:eqHp}
  H^\prime (x) = \int_{i \infty}^x \cosh (x-t) \left \{ - \frac{H(t)}{4 t^2} +
\frac{H^2 (t) }{2 \sqrt{t} } \right \} dt 
-\int_{i \infty}^x \sinh (x-t) \frac{7\cdot 392}{2\cdot 625 t^{9/2}} dt,
\end{equation}
where the last term is the result of integration by parts.
Since $|H(t) |
\le t^{-5/2} \| H \|$, $| \sinh (x-t) |, | \cosh (x-t) | \le 1 $, 
for $x, t \in \Omega_I $
in \eqref{eq:eqHp}, the result follows.
\end{proof}
\begin{Remark}{\rm 
In order to obtain small error bounds for $H$ and $H^\prime$ at $x= 
x_0 := i \left (24\cdot 1.7 \right )^{5/4}/30
= i 3.437\cdots$ corresponding to $z=z_0 = e^{i \pi/5} 1.7$, a good choice 
is $\rho=|x_0| = 3.437\cdots$ and 
$\epsilon = \frac{1}{40}$ (for which the conditions in 
(\ref{0.6}) in Lemma \ref{lem0} hold). With this choice, 
$|H(x_0)| \le |x_0|^{-5/2} \| H \| \le\frac{41}{40} |x_0|^{-5/2}  \|H_0 \| =\frac{41}{40 } \frac{784}{3125}|x_0|^{-5/2} $. 
Also, 
Lemma \ref{lem01} is applied to bound $H^\prime (x_0) $ by using $\|H \| \le \frac{(41)(784) }{(40)(3125)}$.
By \eqref{eq:p1}, \eqref{0.1} and \eqref{1.0a} we have
\begin{equation}
\label{eqy}
y(z) = i\sqrt{\frac{z}{6} } \left ( 1 - \frac{4}{25 x^2} + 
\frac{H(x)}{\sqrt{x}} \right ) \\,   
x= e^{i \pi/4} \frac{(24 z)^{5/4}}{30} ,  
\end{equation}
Defining
$$ y_0 (z) = i \sqrt{\frac{z}{6}} \left ( 1- \frac{4}{25 x^2} \right ), $$ 
straightforward calculations show that
$$ y (z_0) - y_0 (z_0) = i \sqrt{\frac{z_0}{6 x_0} } H(x_0), $$ 
$$ y^\prime (z_0) - y_0^\prime (z_0) = i \sqrt{\frac{z_0}{6}} 
\left ( \frac{1}{z_0 \sqrt{x_0}} \right )
\left ( - \frac{H (x_0)}{8} 
+ \frac{5}{4} x_0 H^\prime (x_0) \right ), $$   
and thus, using the bounds obtained for $H(x_0)$ and $H^\prime (x_0)$, we get
\begin{equation}
\label{erry0z0}
\Big |y (z_0) - y_0 (z_0) \Big | \le \frac{3}{890} , 
~~~\Big | y^\prime (z_0) - y_0^\prime (z_0) \Big |
\le \frac{29}{4468} 
\end{equation}
Note also that 
\begin{equation}
\label{eqy0z0}
y_0 (z_0) = i \sqrt{\frac{z_0}{6} } \left ( 1 - \frac{4}{25 x_0^2} \right ) =
C_1 e^{-2 i \pi/5},\ C_1=-0.5394994\cdots
\end{equation}
\begin{equation}
\label{eqy0z0p}
y_0^\prime (z_0) 
= i \sqrt{\frac{6}{z_0}} \left ( \frac{1}{12} + \frac{4}{75 x_0^2} \right ) = 
C_2e^{i 2 \pi/5}  ,\ C_2=0.148075\cdots      
\end{equation}
}
\end{Remark}
\begin{Remark}\label{SymR}
The function $h$ is real-valued in $\Omega_{I}$. Indeed, 
  with $x=i\tau$,  \eqref{0.1} becomes
  \begin{equation}
    \label{eq:eqsym2s}
    h^{\prime \prime} + \frac{1}{\tau} h^\prime  +h +\frac{h^2}{2} +\frac{392}{625 \tau^4} = 0
  \end{equation}
By complex conjugation symmetry, if $h(\tau)$ is a solution of \eqref{eq:eqsym2s} then so 
is $\overline{h(\overline{\tau})}$; by uniqueness of  the solution
satisfying $h(x) = O(x^{-4})$ for large $x \in \Omega_I$, 
$h(\tau)$ is real-valued. 
\end{Remark}

\section{The tritronqu\'ee in the region $x\in \Omega_1 \cup \Omega_2$}
Consider the domains
$$ \Omega_1 := \left \{ x: ~|x | \ge \rho_0 ~~,~~\frac{\pi}{4} \le \arg x  \le  \frac{\pi}{2} \right \};\ 
\Omega_2 : = \left \{ x: ~~ -\frac{\pi}{4} \le \arg x \le \frac{\pi}{4} ~~,
~~\Re ~x \ge \frac{\rho_0}{\sqrt{2}} 
\right \}, $$   
and 
define the  Banach space $\mathcal{S}_2$ of analytic functions
in the interior of $\Omega_1 \cup \Omega_2$,  
continuous up to the boundary, equipped with the weighted norm
$$ \| H \| = \sup_{x \in \Omega_1 \cup \Omega_2} \Big | x^{5/2}  H \Big | $$
We consider the operator $\mathcal{N} $, defined in
(\ref{0.4}), now acting  on $\mathcal{S}_2$. 
\begin{Lemma}
\label{lemOmega1}
For $\rho_0 \ge \frac{1}{30} \left ( 24\cdot 1.7 \right )^{5/4} =|x_0| $, 
there exists a unique solution 
to the integral equation (\ref{0.4}), $H = \mathcal{N} \left [ H \right ]$,
in $\mathcal{S}_2$, corresponding to 
the tritronqu\'ee through the transformation (\ref{eqy}). 
\end{Lemma}
\begin{proof}
We first obtain bounds on
\begin{equation}
\label{0.8}
H_0 (x) = \int_{i \infty}^x \sinh (x-t) \frac{392}{625 t^{7/2}} dt 
= \int_{i \infty}^x e^{x-t} ~\frac{196}{625 t^{7/2}} dt  
- \int_{i \infty}^x e^{t-x} ~\frac{196}{625 t^{7/2}} dt  
\end{equation}
Consider first $x \in \Omega_2$,
The contour in the middle integral in (\ref{0.8}) can  be deformed
to a radial
one joining $\infty$ to 
$x \in \Omega_1 \cup \Omega_2$ for which $| e^{x-t} | \le 1$ implying
\begin{equation}
  \label{eq:int1}
  \Big | x^{5/2} 
\int_{\infty}^x e^{x-t} \frac{1}{t^{7/2}} dt \Big |  \le \frac{2}{5}
\end{equation}
For the last integral in (\ref{0.8}), 
no such radial path deformation is possible  because of growth
of $e^t$.
On the vertical integration path,
we paramerize $t=x+i \tau$, $\tau\in \mathbb{R}^+$. If we separate
out the real and imaginary parts of $x$:
$x=a+i b$, then $x\in \Omega_2$ implies  
$ \frac{|x|}{\sqrt{2}} \le \Re~x = a $ and
$ |b| \le a$. Then, 
\begin{multline}
\Big | \int_{i \infty}^x e^{t-x} ~\frac{1}{t^{7/2}} dt \Big | 
\le \Big | \int_{0}^\infty 
\left ( a^2 + (b+\tau)^2 \right )^{-7/4} d\tau \Big | 
\\
\le \frac{1}{a^{5/2} } 
\int_{-b/a}^\infty \left ( 1 + p^2 \right )^{-7/4}  dp
\le \frac{1}{a^{5/2} } \int_{-1}^\infty \left ( 1 + p^2 \right )^{-7/4}  d p    
\le \frac{2^{5/4}}{|x|^{5/2}} \int_{-1}^\infty \left ( 1+p^2 \right )^{-7/4} 
dp   
\end{multline}
Combining with \eqref{eq:int1} 
it follows that for $x \in \Omega_2$ we have
$$ \Big | x^{5/2} H_0 (x) 
\Big | \le \frac{196}{625} 
\left [ 2^{5/4} \int_{-1}^\infty (1+p^2)^{-7/4} dp 
+ \frac{2}{5} \right ]  \le \frac{32}{25} 
$$ 
For $x \in \Omega_1$, we note that $|t| \ge |x|$, and on a vertical contour
of integration $|\cosh (x-t) | \le 1$ and $|dt| \le \sqrt{2} d|t|$ and therefore
$$ \Big | H_0 (x) \Big | 
\le \sqrt{2} \int_{|x|}^{\infty} \frac{392}{625 |t|^{7/2} } d|t| 
\le \frac{784 \sqrt{2}}{3125 |x|^{5/2}} $$     
This is clearly a smaller bound than the one for $x\in \Omega_2$.
Therefore, for any $x \in \Omega_1 \cup \Omega_2$, we have
$$ \| H_0 \| \le \frac{32}{25} =: M $$
For  the nonlinear term, 
the calculations are similar. 
For $x\in \Omega_1$, on the vertical path, $x-t\in i\RR$,  $|\sinh (x-t) |\le 1$,
$|H(t) | \le |t|^{-5/2} \| H \|$ and
$|dt| \le \sqrt{2} d|t|$, 
$$ \Big | \int_{i \infty}^x \sinh (x-t) \frac{H^2 (t)}{2 \sqrt{t}} dt\Big| \le
\frac{\sqrt{2}}{9 |x|^{9/2}}  \| H \|^2 $$
For 
$x\in \Omega_2$, we split 
$\sinh$ into two exponentials and break the integral accordingly; 
in one of the integrals the contour can be deformed into a radial path. In the other, we parametrize
the vertical path as in the estimates of $H_0$.
Since the bound in $\Omega_1$ is clearly smaller, this 
results in
\begin{equation}
\label{0.11} 
\left\| \int_{i \infty}^x \sinh (x-t) \frac{H^2 (t)}{2 \sqrt{t} } dt \right\| 
\le N \rho_0^{-2} \| H \|^2, 
\end{equation}
where 
\begin{equation}
\label{0.12}
N = \frac{1}{18} + 
2^{1/4} \int_{-1}^\infty \left (1 + p^2 \right )^{-11/4} dp
\le \frac{203}{138} 
\end{equation}
Now consider the linear term
$$ -\int_{i\infty}^x \sinh (x-t)  \frac{H (t) }{4 t^2} dt $$
For $x \in \Omega_1$, as before, we obtain using $|\sin (x-t) |\le 1$ on a vertical path,
$|dt| \le \sqrt{2} d|t|$ and $|H(t) | \le \|H \| |t|^{-5/2} $, 
$$ \left | \int_{i\infty}^x \sinh (x-t)  \frac{H (t) }{4 t^2} dt \right |
\le \frac{\sqrt{2}}{14 \rho_0 |x|^{5/2}}  \| H \| $$ 
For $x \in \Omega_2$, 
similarly writing $\sinh (x-t)$ in exponential form, breaking the integral accordingly and separately estimating each term we get
$$ \left | \int_{i\infty}^x \sinh (x-t)  \frac{H (t) }{4 t^2} dt \right |
\le \frac{L}{\rho_0 |x|^{5/2}} \| H \|, $$ 
$$ L = \frac{1}{28} + 2^{-5/4} \int_{-1}^\infty (1+p^2)^{-9/4} dp   
\le \frac{3}{5} 
$$
Clearly, the bound for $x\in \Omega_2$ is larger. We conclude that
$$ \left\| 
\int_{i\infty}^x \sinh (x-t)  \frac{H (t) }{4 t^2} dt \right\|
\le \frac{L}{\rho_0} \| H \|$$ 
$\mathcal{N}: B \rightarrow B$ is
  contractive if
for some $\epsilon > 0$ we have
\begin{equation}
L \rho_0^{-1} (1+\epsilon) +
N \rho_0^{-2} M (1+\epsilon)^2 \le \epsilon  ~~,~~ L \rho_0^{-1} +
2 N \rho_0^{-2} M (1+\epsilon) < 1      
\end{equation}
Both conditions are  satisfied for $\rho_0=|x_0|$, when 
$\epsilon =\frac{3}{2} $, {\it i.e.} ball size is
$\frac{5}{2} M $,
and the lemma follows from the
contraction mapping theorem, 
and the fact that $\Omega_I$ is the boundary of $\Omega_1$ implies 
that this solution is the same
in Lemma \ref{lem0}, 
--corresponding to the tritronqu\'ee.
\end{proof}

\section{Analysis of $y_t$ for $x\in\Omega_3$}

\begin{Definition}
\label{defOmega4}
Let
$$ \Omega_3: = \left \{ x: ~|x| \ge \rho \ge 3 ~~,~~\arg x  \in \left [-\frac{\pi}{2},
-\frac{\pi}{4} \right ] \right \}  
$$
\end{Definition}
Since $\Omega_3$ is close to the region of poles, 
 the power series asymptotic representation does not suffice anymore;  one needs to include a few exponentially small terms of the transseries.
The results in \cite{invent} imply that for large $x$ in $\Omega_3$ 
s.t.  $\Big |\frac{S e^{-x}}{\sqrt{x}} - 12 \Big | \ge m > 0$, $h$ has the uniform asymptotic expansion 
to all orders \footnote{See also \cite{Kapaev-2004} for a derivation of the leading order asymptotics using the Riemann-Hilbert representation.}
$$ h(x) \sim \sum_{j=0}^\infty x^{-j} F_j (\xi), $$ 
where 
\begin{equation}
\label{5}
\xi = \frac{S e^{-x}}{\sqrt{x}}, ~~S = i \sqrt{\frac{6}{5 \pi}}  
\end{equation}
and the $F_j (\xi)$'s are
rational functions with poles at $\xi=12$. The value of $S$ is obtained from
the constant $s_{4-2l}e^{-i\pi/4}/\sqrt{10\pi}$ of eq. (19) in  \cite{KitaevKapaev} after taking into account the needed changes of variables. 
The two scale expansion in \cite{invent} is general and, to the best of our knowledge, was introduced there for the first time.

What matters most in Theorem \ref{Thm1} is that $|S|=0.61804 \cdots$ is fairly small implying
$\xi$ is small for $x \in \Omega_3$ even for
$\rho = 3$.  Therefore, a few terms in the Taylor series of each $F_j$ at
$\xi=0$ and a fairly small number of $F_j$'s are expected to yield  a good approximation of
$h(x)$. 

With this expectation, we choose an approximate expression $h_0$ for $h_t$ 
in $\Omega_3$ in the following form (cf. \cite{invent} page 38)
\begin{equation}
\label{1.0}
h_0 (x) = \left ( \xi + \frac{\xi^2}{6} + \frac{\xi^3}{48} +
\frac{\xi^4}{432} + \frac{5 \xi^5}{20736}  \right ) + 
\frac{1}{x} 
\left ( - \frac{\xi}{8} - \frac{11}{72} \xi^2 - \frac{43}{1152} \xi^3 \right )
+\frac{9 \xi}{128 x^2} , 
\end{equation}
From this point on, the analysis is similar in spirit to that in the previous sections: formulating an integral
equation for the error term $h-h_0$ and using contractive mapping arguments to control it. 
However, given the shape of the error term, bounding the integral terms becomes slightly more involved. 
Let
\begin{equation}
\label{2}
h(x) = h_0 (x) + x^{-1/2} G (x) 
\end{equation}
By  (\ref{1}), 
$G(x)$ satisfies
\begin{equation}
\label{3}
G^{\prime \prime} - \left ( 1 + h_0 (x) \right ) G(x) = \frac{G^2 (x)}{2 \sqrt{x}} - 
R(x) - \frac{G(x)}{4 x^2}, 
\end{equation}
where  $R(x)$ is given by 
\begin{equation}
\label{3.1}
R(x) = \sqrt{x} \left ( h_0^{\prime \prime} +\frac{1}{x} h_0^\prime - h_0 (x) - \frac{1}{2} h_0^2 (x) - \frac{392}{625 x^4} 
\right ) 
= \sum_{j=5}^{9} x^{-j/2} r_{j} \left (e^{x} \right ) 
\end{equation}
and the  $r_j (\zeta)$ are polynomials in $\zeta^{-1}$, where only $r_7$ has a 
nonzero constant term $-\frac{392}{625}$ (see \eqref{3.1.1}-\eqref{3.1.5} 
in the Appendix  for the precise 
expressions of $r_j$).  
Define
\begin{equation}
\label{4.0}
R_0 (x) = x^{-5/2} r_5 (e^{x} )  ~~,~~R_1 (x) = x^{-3} ~r_6 (e^{x})  ~,
~{\tilde
 R} (x) =
\sum_{j=7}^9 x^{-j/2} r_j (e^{x}) = R-R_0 - R_1
\end{equation}
and
\begin{equation}
\label{4.0.0}{\tilde r}_7 (\zeta) = r_7 (\zeta) + \frac{392}{625} ~~,~~~{\tilde r}_j (\zeta) 
= r_j (\zeta) ~~
{\rm for} ~j \ne 7
\end{equation}
To write \eqref{3} in integral form, we  need the 
properties of the  Green's function of the operator on the
left side of the equation. 
It is more convenient to write a 
nearby equation with an explicit Green's function, and for this end we find quasi-solutions of the homogeneous equation
\begin{equation} 
\label{4}
u^{\prime \prime} - \left ( 1 + h_0 (x) \right ) u
=0
\end{equation} 
Formal asymptotic arguments for large $x$ suggest that 
one solution of (\ref{4}) has the asymptotic behavior
\begin{equation}
\label{5.0}
u \sim y_1 (x) = e^{-x} \left ( 1 + \frac{J(x)}{\sqrt{x}} \right ),  
\end{equation}
where
\begin{equation}
\label{5.01}
J(x) = \frac{S e^{-x}}{3} 
+\frac{S^2 e^{-2x} }{16 \sqrt{x}} 
- \frac{19 S e^{-x}}{72 x}   
+ \frac{S^3 e^{-3x}}{108 x} - \frac{5 S^2 e^{-2x}}{48 x^{3/2} } 
+ \frac{25 S^4 e^{-4x}}{20736 x^{3/2}}  
\end{equation}
We can readily check that $y_1$ solves \eqref{4} up to $O (x^{-5/2}) $ errors:
\begin{equation}
\label{6}
y_1^{\prime \prime} - \left ( 1 + h_0 (x) \right ) y_1  
= q (x) y_1, 
\end{equation}
where
\begin{equation}
\label{7}  
q (x)  y_1 (x) = \sum_{j=5}^{9} x^{-j/2} q_j \left ( e^{x} \right ),  
\end{equation}
Here, all $q_j (\zeta)$ are polynomials in $1/\zeta$ 
of degree at least 2
(see equations \eqref{7.1}-\eqref{7.5} in the Appendix).
We chose $y_1$ to 
ensure that the error term  $q y_1$ is $ O(x^{-5/2} ) $ for large $x \in \Omega_3$.

A second independent solution to the homogeneous equation (\ref{6}) $y_2$ is given by  
$ y_1 (x) \left ( \int^x \frac{dx'}{y_1^2 (x') }  \right )$. With a 
suitable  choice of
integration constant, $y_2$ becomes
\begin{equation}
\label{7.01}
y_2 (x) = y_1 (x) \left [ \frac{5 S^2}{24} \log (i x) + z_2 (x) \right ], 
~~~{\rm where} ~~z_2 (x) = z_{2,0} (x) + z_{2,1} (x) + z_{2,R} (x),
\end{equation}
Here
\begin{equation}
\label{7.1.0}
z_{2,0} (x) =
\frac{e^{2x}}{2} ~~~,~~~
z_{2,1} (x) =  
- \frac{2 S e^{x}}{3 \sqrt{x}}  
\end{equation}
\begin{equation}
\label{7.2.0}
z_{2,R} (x) = 
\int_{-i \infty}^{x} dx' \left ( \frac{1}{y_1^2 (x')} - 
e^{2x'} +\frac{2S e^{x'}}{3 \sqrt{x'}} - 
\frac{S e^{x'}}{3 {x'}^{3/2} } 
- \frac{5 S^2}{24 x'}   
\right )
\end{equation}
Using the fact that $y_1$ and $y_2$ defined above solve (\ref{6})
and their Wronskian is $y_1 y_2^\prime - y_2 y_1^\prime = 1$, inversion of (\ref{3})
results in the
integral equation
\begin{equation}
\label{9}
G(x) = G_0 (x) + 
\int_{-i\infty}^x \left [ y_2 (x) y_1 (x') - y_1 (x) y_2 (x') \right ] \left [ 
-V(x') G(x') + 
\frac{G^2 (x')}{2 \sqrt{x'}} \right ]dx' =: \mathcal{N} \left [G \right ] (x) 
\end{equation}
where
\begin{equation}
\label{9.1}
V(x) = V_0 (x) + q (x) \ , 
~~V_0 (x) = \frac{1}{4 x^2}  
\end{equation} 
\begin{equation}
\label{10}
G_0 (x) = 
\int_{-i\infty}^x \left [ y_1 (x) y_2 (x') - y_2 (x) y_1 (x') \right ] R(x') 
dx'
\end{equation}
\begin{Definition}{\rm 
Let  $\mathcal{S}_4$ be the Banach space of functions analytic in the interior of $\Omega_3$ and 
continuous on $\overline{\Omega_3}$, equipped with the norm
\begin{equation}
\label{10.0.1}
\| G \| = \sup_{x \in \Omega_3} \Big | x^{5/2} G (x) \Big |  
\end{equation}
The usual sup norm will be denoted by $\| . \|_{\infty}$.}
\end{Definition}
We seek a solution to (\ref{9}), {\it i.e.} $G = \mathcal{N} \left [ G
\right ]$ in $\mathcal{S}_4$ with $\rho \ge 3$.  It will
be proved that  $\mathcal{N}:\mathcal{S}_4\to\mathcal{S}_4  $, see (\ref{9}).  
The integral reformulation of (\ref{3}) is 
\begin{equation}
\label{10.1}
G(x) = C_1 y_1 (x) + C_2 y_2 (x) + \mathcal{N} \left [ G \right ] (x) 
\end{equation}
with $(C_1, C_2)=(0,0)$ 
since neither $y_1$ nor $y_2$ are in $\mathcal{S}_4$.  Therefore,
any solution to (\ref{3}) in $\mathcal{S}_4$ must necessarily satisfy (\ref{9}).

We now prove the following result:
\begin{Theorem}
\label{Thm2}
For $\rho \ge 3$, there exists a unique solution $G$ to
(\ref{9}) in a ball $B_{4}\subset\mathcal{S}_4 $ of radius 4.
Through the change of variables in \eqref{0.1} and \eqref{2}, $G$ corresponds to the tritronqu\'ee solution $y=y_t$.
\end{Theorem}

The proof of Theorem \ref{Thm2}  follows from standard contractive mapping arguments,  using the estimates in 
Lemmas \ref{lem1}-\ref{lem3}. These lemmas are proved in subsections \ref{lem2sub}-\ref{lem1sub}.
\begin{Lemma}
For $\rho \ge 3$ we have
\label{lem1}
\begin{equation}
\label{10.1.1}
 \| G_0 \| \le 2 
\end{equation}
\end{Lemma}
\begin{Lemma}
For $\rho \ge 3$, and $G \in \mathcal{S}_4$,
\label{lem2}
\begin{equation}
\label{10.2}
\L\| \int_{-i \infty}^x \left [ y_2 (x) y_1 (x') - y_1 (x) y_2 (x') \right ] V(x') G(x') dx' \R\| 
\le \frac{1}{4} \| G \| 
\end{equation}
\end{Lemma}
\begin{Lemma}
\label{lem3}
For $\rho \ge 3$, and $G, G_1, G_2 \in \mathcal{S}_4$,
\begin{equation}
\label{10.3.0}
\L\| \int_{-i \infty}^x \left [ y_2 (x) y_1 (x') - y_1 (x) y_2 (x') \right ] 
\frac{1}{2 \sqrt{x'} } G^2 (x') dx' \R\| 
\le \frac{1}{25} \| G \|^2 , 
\end{equation}
\begin{multline}
\label{10.3}
\L\| \int_{-i \infty}^x \left [ y_2 (x) y_1 (x') - y_1 (x) y_2 (x' ) \right ] \frac{1}{2 \sqrt{x'} } 
\left [ G_1^2 (x') - G_2^2 (x') \right ]  dx' \R\|  \\
\le \frac{1}{25} \left (\| G_1  \| + \| G_2 \| \right ) \| G_1 - G_2 \|   
\end{multline}
\end{Lemma}

\noindent{\bf Proof of Theorem \ref{Thm2}.} 
It is clear from Lemmas \ref{lem1},
\ref{lem2} and \ref{lem3} that $\mathcal{N}:B_4\to B_4$  (cf. (\ref{9})): 
 for $G$, $G_1$ and $G_2$ in $B_4$ we have
$$\| \mathcal{N} \left [ G \right ] \| \le 
\| G_0 \| + \frac{1}{4} \| G \| + \frac{1}{25} \| G \|^2 \le 2 
+ \frac{4}{4} + \frac{16}{25} < 4. 
$$
and 
$$ \| \mathcal{N} \left [ G_1 \right ] - \mathcal{N} \left [ G_2 \right ]  \| 
\le \frac{1}{4} \| G_1 - G_2 \| + 
\frac{8}{25} \| G_1 - G_2 \| \le \frac{3}{4} \| G_1 - G_2 \|      
$$
By the contraction mapping theorem, there is a unique solution to
  (\ref{9}) in $B_4$ if $\rho \ge
  3$. From 
\eqref{2}, it is clear that  $G$ corresponds to 
a solution $h$ of \eqref{1} that is singularity-free in the closed
domain $\Omega_3$ and has the leading order
asymptotic behavior $h \sim \frac{S e^{-x}}{\sqrt{x}} $ 
as $x \rightarrow -i \infty$ on
the negative imaginary axis. 
By \cite{invent} \S 5.2 (see also \cite{duke}), for
any $C$ there is a unique solution $h$ of \eqref{1} with the behavior
$Ce^{-x}x^{-1/2}$ as $x\to -i\infty$ analytic for large
$x$ in a sector in the fourth quadrant.\footnote{The analysis in
      \cite{invent} is done in the first quadrant, but by symmetry
      w.r.t. $x\in\RR$ it applies with straightforward modifications
      to the fourth quadrant.} 
The value $C=S$ identifies this
solution with the tritronqu\'ee. (This also follows from classical results cf. \cite{KitaevKapaev}.)
      
\subsection{Preliminary Lemmas}
In this subsection
we obtain various integral estimates needed in the
sequel.
\begin{Definition}{\rm Let $P$ be a polynomial, $P (\eta)= \sum_{m=m_0}^{m_1} p_m \eta^{-m}$. 
\label{defFunctional}
We define the following weighted $\ell^1$ norms:
$$ \mathcal{F}_{1,j} \left [P \right ] = \frac{2}{j-2} \sum_{m=m_0}^{m_l} |p_m| 
~{\rm for} ~j > 2
;\  \mathcal{F}_{2,j} \left [P \right ] = \sum_{m=m_0}^{m_l} \frac{2}{m} |p_m| 
~{\rm for} ~m_0 > 0 $$
$$ \mathcal{F}_{3,j} \left [P \right ] = \frac{2}{j-3} \sum_{m=m_0}^{m_l} |p_m| 
~{\rm for} ~j > 3;\  \mathcal{F}_{4,j} \left [P \right ] = \sum_{m=m_0}^{m_l} 
\frac{j^2+2j-2}{j(j-1) m} |p_m| ~~{\rm for} ~j > 1 ~,~m_0 > 0 $$  
}\end{Definition}

\begin{Lemma}
\label{lem1prel}
If $l_0 > 2$, $g$ is analytic in $\Omega_3$ with $ \| g \|_\infty < \infty$, and
\begin{equation}
\label{10.1.0.0}
w(x) = \sum_{l=l_0}^{L} x^{-l/2} P_l \left (e^{x} \right ), 
\end{equation}
where $P_l (\zeta) = \sum_{m=0}^{m_l} p_{m,l} \zeta^{-m} $ then
\begin{equation}
\label{10.1.0.0.new}
\Big | \int_{-i \infty}^x g(x') w(x') dx' \Big | 
\le  
\| g \|_\infty \sum_{l=l_0}^{L} |x|^{-l/2+1}  Q_l ,
\end{equation} 
where $Q_l = \mathcal{F}_{1,l} \left [ P_l 
\right ]$ 
\end{Lemma}
\begin{proof}
The various terms  in the integrand 
in (\ref{10.1.0.0.new}) are of the form $ g(x') p_{m,l} e^{-mx'} {x'}^{-l/2} $ 
and thus the contour of
integration can  be deformed to a radial path from $\infty$ to $x \in \Omega_3$. Note also that
$|e^{-mx} g(x) | \le \| g \|_\infty$ for $x\in \Omega$. Then, clearly,
$$
\Big | p_{m,l} \int_{\infty}^x g(x') e^{-mx'} {x'}^{-l/2} d x \Big |
\le |p_{m,l} | \| g \|_\infty 
\int_{\infty}^{|x|} |x'|^{-l/2} d|x'| \le \frac{|p_{m,l}| |x|^{-l/2+1}}{l/2-1} 
$$
\end{proof}
\begin{Lemma}
\label{lem2prel}
If $l_0 > 0$ and 
\begin{equation}
\label{10.1.0.0.0}
w(x) = \sum_{l=l_0}^{L} x^{-l/2} P_l \left (e^{x} \right ), 
\end{equation}
where $P_l (\zeta) = \sum_{m=1}^{m_l} p_{m,l} \zeta^{-m} $ is a polynomial 
in $\zeta^{-1}$, then
\begin{equation}\label{eq50}
\Big | \int_{-i \infty}^x w(x') dx' \Big | 
\le  \sum_{l=l_0}^{L} Q_l x^{-l/2}  
\end{equation}
where $Q_l = \mathcal{F}_{2,l} \left [
P_l \right ]$   
\end{Lemma}
\begin{proof}
We note that the terms in the integrand in \eqref{eq50} are of the form
$$ p_{m,l} e^{-mx'} {x'}^{-l/2} 
= \frac{d}{dx'} \left ( -\frac{p_{m,l}}{m} e^{-m x'} {x'}^{-l/2} \right ) 
- \frac{l p_{m,l}}{2 m} e^{-mx'}  {x'}^{-l/2-1} 
$$
Therefore, integrating out the first term on the right hand side above
explicitly and deforming the path for the second term to a radial contour, it follows that 
$$ \Big | p_{m,l} \int_{\infty}^x e^{-mx'} {x'}^{-l/2} dx' \Big | 
\le 
\frac{ 2|p_{m,l}|}{m} {|x|}^{-l/2} 
$$
\end{proof}
\begin{Lemma}
\label{lem3prel}
If for $l_0 > 3$ we have   
$ w(x) =  \sum_{l=l_0}^{L} x^{-l/2} P_l (e^{x} )$,   
where $P_{l} (\zeta) = \sum_{m=0}^{m_l} p_{m,l} \zeta^{-m} $, and if $g$ is analytic in $\Omega_3$ with
$\| g \|_\infty < \infty$, then
\begin{equation}
\label{eq1lem3prel}
\Big | \int_{-i \infty}^x \log \left ( \frac{x'}{x} \right )~g(x') w(x') dx' \Big |  
\le \| g \|_{\infty} \sum_{l=l_0}^L Q_l |x|^{-l/2+1}, 
\end{equation}
where $Q_l = \mathcal{F}_{3,l} 
\left [ P_l \right ]$ 
\end{Lemma}
\begin{proof}
Once again because of the analyticity and decay of the integrand in
\eqref{eq1lem3prel}, 
we may deform the integration path to a radial one  
joining $\infty$ to $x \in \Omega_3$. The  general term in the integrand is of the form
$$ p_{m,l} {x'}^{-l/2} e^{-mx'} g(x') \log \left ( \frac{x'}{x} \right ) $$    
Since it is readily checked that for $|x'| \ge |x|$, 
$$ \Big | \log \left ( \frac{x'}{x} \right ) \Big | 
\le \left | \frac{x'}{x} \right |^{1/2}, $$ 
it follows that  
\begin{multline}
 \Bigg | \int_{\infty}^x p_{m,l} {x'}^{-l/2} e^{-mx'} g(x') \log \left ( \frac{x'}{x} \right ) dx' \Bigg |     
\\\le \| g \|_\infty |p_{m,l}| |x|^{-1/2}  \int_{\infty}^{|x|} |x'|^{-l/2+1/2} d|x'| 
= \| g \|_\infty \frac{2 |p_{m,l}|}{l-3} |x|^{-l/2+1}   
\end{multline}
\end{proof}
\begin{Lemma}
\label{lem4prel}
If for $l_0 > 1$,
$ w(x) =  \sum_{l=l_0}^{L} x^{-l/2} P_l (e^{x} )$,   
where $P_{l} (\zeta) = \sum_{m=1}^{m_l} p_{m,l} \zeta^{-m} $, 
then
$$ \Bigg | \int_{-i \infty}^x \log \left ( \frac{x'}{x} \right ) w(x') dx' \Bigg |  
\le \sum_{l=l_0}^L Q_l |x|^{-l/2}, $$ 
where $Q_l = 
\mathcal{F}_{4,l} \left [ P_l \right ]$ 
\end{Lemma}
\begin{proof}
Once again because of the analyticity and decay of the integrand, we may deform the integration path to a radial one  
joining $\infty$ to $x \in \Omega_3$. A general term in the integrand is of the form
$$ p_{m,l} {x'}^{-l/2} e^{-mx'} \log \frac{x'}{x} = 
\frac{d}{dx'} \left ( - \frac{p_{m,l}}{m} e^{-mx'} {x'}^{-l/2} \log \frac{x'}{x}  
\right ) + \frac{p_{m,l}}{m} e^{-m x'} {x'}^{-l/2-1} $$ 
$$-\frac{l}{2m} p_{m,l} e^{-mx'} {x'}^{-l/2-1} \log \frac{x'}{x} $$        
Noting that the complete derivative is zero at the end point $x'=x$ and 
applying Lemma \ref{lem3prel} to bound the 
integral of the last term (on a radial path) we immediately obtain
$$ \Big | \int_{\infty}^x p_{m,l} e^{-mx'}  
\log \frac{x'}{x} dx' \Big | \le 
\frac{1}{m} |p_{m,l}| \left ( \frac{2}{l} + \frac{l}{l-1} \right )  |x|^{-l/2}, 
$$
from which the Lemma follows.
\end{proof}

\subsection{
Bounds on $J$, $y_1$, $z_{2,R}$ and $ z_2$ for $x \in \Omega_3$} 

Let
\begin{equation}
\label{5.01.0}
y_{1,0} (x) = e^{-x} ~~, 
~~y_{1,1} (x) = \frac{S}{3\sqrt{x}} e^{-2 x} ~~,
~~y_{1,R} = y_1 - y_{1,0} - y_{1,1},  
\end{equation}
and define
\begin{equation}
\label{5.01.1}
j(x) =\frac{3 S}{16} {{\rm e}^{-x}}+ \left ( - \frac{19}{24} + \frac{S^2 e^{-2x}}{36} 
\right ) x^{-1/2}
+ \left( -{\frac {5}{16}}\,S{{\rm e}^{-x}}+{
\frac {25}{6912}}\,{S}^{3}{{\rm e}^{-3\,x}} \right) {x}^{-1}
\end{equation}
Comparing with (\ref{5.01}) we see that
\begin{equation}
\label{5.01.2}
J(x) = \frac{S e^{-x}}{3} \left ( 1 + \frac{j(x)}{\sqrt{x}} \right ) 
\end{equation}
Note that for $x \in \Omega_3$, 
\begin{equation}
\label{5.01.3}
|j(x) | \le 
\frac{3 |S|}{16} + 
{\frac {19}{24}}\,{\frac {1}{\sqrt {\rho}}}
+\frac{|S|^2}{36\sqrt{\rho}}
+{\frac {5}{16}}\,{\frac { \left| S \right| }{\rho}}+{\frac {25 |S|^3}{6912 \rho}}
=: j_m 
\end{equation}
Using $J(x)$ in (\ref{5.01.2}), it follows that for $x\in \Omega_3$ we have
\begin{equation}
\label{12.0}
\Big |e^{x} J (x) \Big | \le  
\frac{|S|}{3} \left ( 1 + \frac{j_m}{\sqrt{\rho}} \right )
=: J_M 
\end{equation}
From \eqref{5.0}, it follows that
\begin{equation}
\label{12.1}
\Big | e^{x} y_1 \Big | \le \left ( 1 + \frac{J_M}{\sqrt{\rho}} \right ) 
=: Y_{1,M} 
\end{equation}
Now,  (\ref{5.01}) and (\ref{5.01.0}) imply
\begin{equation}
\label{12.2}
\Big | x e^{2 x} y_{1,R} (x) \Big |
\le \frac{|S|}{3} j_m =: Y_{1,R,M} 
\end{equation}
From (\ref{5.01.0}), it also follows that
for $x \in \Omega_3$,
\begin{equation}
\label{12.3}
\Big | e^{x} \left [ y_{1,0} (x) + y_{1,1} (x) \right ] \Big |  
\le 1 + \frac{|S|}{3 \sqrt{\rho}}  
\end{equation} 
Expressing $y_1$ in terms of $J$, as in \eqref{5.0} and \eqref{5.01}, 
in  \eqref{7.2.0}, it follows that
\begin{multline} 
\label{12}
z_{2,R} (x) = z_{2,R,0} (x) 
+ \int_{-i\infty}^x E (x')  dx'   
+\frac{7S}{36} \int_{-i \infty}^x  \frac{e^{x'}}{{x'}^{3/2}} dx' \\
+  
\int_{-i\infty}^x e^{2 x'} \left [ \left ( 1 + \frac{J(x')}{\sqrt{x'}} \right )^{-2}      
-1 + \frac{2 J(x')}{\sqrt{x'}} - \frac{3 J^2 (x')}{x}  \right ] dx' ,   
\end{multline}
where
\begin{equation}
\label{12.0.0}
z_{2,R,0} (x) =
\frac{23 S^2}{72 x} - \frac{361 S^2}{3456 x^2}  
-\frac{23 S^3 e^{-x}}{216 x^{3/2}} - \frac{577 S^4 e^{-2x}}{41472 x^2} ~,~   
\end{equation}
\begin{equation}
\label{13}
E(x) = \sum_{j=5}^8 x^{-j/2} E_j \left ( e^{x} \right ), 
\end{equation}  
In (\ref{13}), each $E_j (\zeta)$ is a polynomial in $1/\zeta$ with no constant
term (the precise expressions are given in \eqref{13.1}-\eqref{13.4} in the 
Appendix).
We note that
\begin{equation}
\label{13.1.0}
\Big | x z_{2,R,0} \Big |
\le \frac{23 |S|^2}{72} + \frac{23 |S|^3}{216 \rho^{1/2}} +
\left ( \frac{361 |S|^2}{3456} + \frac{577 |S|^4}{41472} \right ) \rho^{-1} 
\end{equation}
Using Lemma \ref{lem2prel},
it follows that 
\begin{equation}
\label{13.2.2}
\Big | x \int_{-i\infty}^x E (x') dx' \Big | 
\le \sum_{j=5}^8 \rho^{-j/2 +1} \mathcal{F}_{2,j} \left [ E_j \right ]  
=: E_M 
\end{equation}
(see Definition \ref{defFunctional} and
the expression of $E_M$ in the Appendix, \eqref{A.1.0}).
On integration by parts, we get
\begin{equation}
\label{13.3.0}
\frac{7S e^{-x}}{36} \int_{-i\infty}^x \frac{e^{x'}}{{x'}^{3/2} } dx'  
= \frac{7S}{36 x^{3/2}} + 
\frac{7S}{24} \int_{-i\infty}^x \frac{e^{x'-x}}{{x'}^{5/2}} dx' 
\end{equation}
Therefore for $x \in \Omega_3$ we get
\begin{equation}
\label{13.4.0}
\Big | \frac{7S x e^{-x}}{36} \int_{-i\infty}^x \frac{e^{x'}}{{x'}^{3/2} } dx'  
\Big |
\le \frac{7|S| \left ( \sqrt{2} + 1 \right )}{36 \sqrt{\rho}} 
\end{equation}
where we used the fact that on a vertical contour connecting $-i\infty$ to $x \in \Omega_3$,
$|dx'| \le \sqrt{2} d|x'|$. 
We note that 
\begin{equation}
\label{13.4.1.0}
e^{2x} \left [ \left ( 1 + \frac{J(x)}{\sqrt{x}} \right )^{-2}\!\!\!\! - 1 + 2 \frac{J(x)}{\sqrt{x}}  - 
\frac{3J^2(x)}{x}  \right]  
= - 4 e^{2x}x^{-\frac32} J^3(x)
+ \frac{5 x^{-2} e^{2x} J^4 (x)}{1+x^{-1/2} J(x)} - 
\frac{x^{-\frac52} e^{2x} J^5 (x)}{(1+x^{-1/2} J(x))^2}, 
\end{equation}
and 
\begin{equation}
\label{13.4.1}
-4 e^{2x} x^{-3/2} J^3 (x) = - \frac{4S^3 e^{-x}}{27 x^{3/2}} 
-\frac{4 S^3}{27 x^{3/2}} e^{-x}    
\left [ \left ( 1 + \frac{j(x)}{\sqrt{x}} \right )^3 - 1 \right ] 
\end{equation}
Now, using Lemma \ref{lem2prel}, it follows that
\begin{equation}
\label{13.4.2.1}
\Big | x \int_{-i\infty}^x \frac{4 S^3 e^{-x'}}{27 {x'}^{3/2}} dx' \Big |
\le 
\frac{8 |S|^3}{27 \rho^{1/2} },  
\end{equation}
Deforming the contour to a radial one, it is clear that 
\begin{equation}
\label{13.4.3}
\bigg | -\frac{4 S^3 x}{27} \int_{-i \infty}^x \frac{e^{-x'}}{{x'}^{3/2}} 
\left [ \left ( 1 + \frac{j(x')}{\sqrt{x'}} \right )^3 - 1 \right ] dx'\bigg | 
\le \frac{4 |S|^3}{27} \left ( 3 j_m + 
3 \frac{j_m^2}{\sqrt{\rho}} + \frac{j_m^3}{\rho} \right )      
\end{equation}
Therefore, 
using (\ref{13.4.2.1}) and (\ref{13.4.3}) in 
(\ref{13.4.1}) we get
\begin{equation}
\label{13.4.4}
\Big | x \int_{-i \infty}^x \frac{-4 e^{2x'} J^3 (x')}{{x'}^{3/2}} dx' \Big |
\le \frac{8 |S|^3}{27 \sqrt{\rho}} + \frac{4 |S|^3 j_m}{9} \left ( 1 + 
\frac{j_m}{\sqrt{\rho}} + \frac{j_m^2}{3 \rho} 
\right )
\end{equation}   
From (\ref{13.4.1.0}) and (\ref{13.4.4}), 
it follows that
\begin{multline}
\label{13.5}
\Big | x \int_{-i\infty}^x 
\left \{ e^{2x'} \left ( 1 + {x'}^{-1/2} J (x')   \right )^{-2} - 1 + 2 {x'}^{-1/2} J (x')  - 
3 {x'}^{-1} J^2 (x') dx' \right \} \Big | \\
\le 
\frac{8 |S|^3}{27 \sqrt{\rho}} + \frac{4 |S|^3 j_m}{9} \left ( 1 + \frac{j_m}{\sqrt{\rho}} + \frac{j_m^2}{3 \rho} 
\right )
+
\frac{5 J_M^4 }{1-\rho^{-1/2} J_M}  +  
\frac{2 \rho^{-1/2} J_M^5 }{3 \left (1-\rho^{-1/2} J_M \right )^2}    
\end{multline}
Therefore, combining 
the estimates (\ref{13.1.0}), (\ref{13.2.2}), (\ref{13.4.0})
and (\ref{13.5}) in the expression (\ref{12}) for $z_{2,R}$, we get
\begin{multline}
\label{13.6}
\Big | x e^{-x} z_{2,R} \Big | \le 
\frac{23 |S|^2}{72} + \frac{23 |S|^3}{216 \rho^{1/2}} +
\left ( \frac{361 |S|^2}{3456} + \frac{577 |S|^4}{41472} \right ) \rho^{-1} 
+ E_M \\
+  
\frac{7|S| \left ( \sqrt{2} + 1 \right )}{36 \sqrt{\rho}} 
+  
\frac{8 |S|^3}{27 \sqrt{\rho}} + \frac{4 |S|^3 j_m}{9} \left ( 1 + \frac{j_m}{\sqrt{\rho}} + \frac{j_m^2}{3 \rho} 
\right ) 
+
\frac{5 J_M^4 }{1-\rho^{-1/2} J_M}  +  
\frac{2 \rho^{-1/2} J_M^5 }{3 \left (1-\rho^{-1/2} J_M \right )^2}    
=: z_{2,R,M} 
\end{multline}
Therefore, using \eqref{7.01} and  \eqref{7.1.0} we get
\begin{equation}
\label{13.7}
|e^{-2x} z_2  | \le \frac{1}{2} + | e^{-x} z_{2,1}|
 + \frac{z_{2,R,M}}{\rho} 
= \frac{1}{2} + 
\frac{2 |S|}{3 \sqrt{\rho}}  
+
\frac{z_{2,R,M}}{\rho} =: z_{2,M}
\end{equation}
To help the reader who would like to check the intermediate 
  steps in the calculations, we provide in the Appendix the numerical
  values for $\rho=3$ of the various constants appearing in the estimates.

\subsection{Bounds on $V(x)$ and proof of Lemma \ref{lem2}}
\label{lem2sub}

We first seek bounds on $q y_1$. 
It is clear from the expression of $q y_1$ in  
(\ref{7}) that 
Lemma \ref{lem1prel} 
applies to $w(x) = x^{-5/2} e^{2x} q y_1 $, $g(x) = x^{5/2} G(x)$.
Noting that $\| G \| = \|x^{5/2} G\|_\infty$, we obtain  
\begin{equation}
\label{14}
\Big | x^{5/2} \int_{-i \infty}^x e^{2x'} q (x') y_1 (x') G(x') dx' \Big |
\le  \| G \| \sum_{j=10}^{14}  \mathcal{F}_{1,j} \left [ q_{j-5} \right ] \rho^{-j/2+7/2}   
= M_q \| G \|  
\end{equation}
(see Def. \ref{defFunctional}).
The explicit formula  of $M_q$ is given in (\ref{A.1.2})  in the Appendix.
Further, 
\begin{multline}
\label{15}
\Big | x^{5/2} y_1 (x) 
\int_{-i \infty}^x \left ( z_2 (x') - z_2 (x) \right ) q (x') y_1 (x') G(x') dx' 
\Big |
\le 2 Y_{1,M} z_{2,M} M_q \| G \| \\
\end{multline}
In \eqref{15}, recalling that $q y_1 $ is a polynomial in $1/e^{x}$ of
degree at least 2, 
we applied Lemma~\ref{lem1prel} with $g(x') = e^{-2 x'} z_2 (x') {x'}^{5/2} 
G(x') $ and
$w(x') = {x'}^{-5/2} e^{2 x'} q (x') y_1 (x')$ in the part of the integral involving $z_2 (x')$, while in the 
second one, we took $g(x') = e^{x-x'} {x'}^{5/2} G(x') $, 
$w(x') = {x'}^{-5/2} e^{x'}  
q_1 (x') y_1 (x') $ and used 
$\Big | e^{-x} z_2 (x) y_1 (x) \Big | \le
z_{2,M} Y_{1,M}$.
Lemma \ref{lem3prel}, 
for $w(x) = x^{-5/2} q (x) y_1 (x) $, $g(x)=x^{5/2} G$, implies
\begin{multline}
\label{15.1}
\Big | \frac{5 S^2}{24} x^{5/2} y_1 (x) \int_{-i \infty}^x ~
\log \left ( \frac{x'}{x} \right )~q (x') y_1 (x') G(x') dx' 
\Big |
\le \frac{5 |S|^2}{24} 
\| G \| \sum_{j=10}^{14} \mathcal{F}_{3,j} \left [ q_{j-5} \right ] \rho^{-j/2+7/2}  
\\
=: \frac{5 |S|^2}{24} M_{L,q}
\| G \| 
\end{multline}
where the detailed expression of  $M_{L,q}$ is given in the 
Appendix, 
(\ref{A.1.6})).  
Therefore, 
using (\ref{7.01}), \eqref{15} and \eqref{15.1},  
it follows that 
\begin{equation}
\label{16}
\left\| \int_{-i\infty}^x \left ( y_2 (x) y_1 (x') - y_1 (x) y_2 (x') \right ) 
q (x') G \right\|
\le Y_{1,M} \left ( 2 z_{2,M} M_q + \frac{5 |S|^2}{24} M_{L,q} \right ) \| G \|
\end{equation}
We now bound the terms involving $V_0 = \frac{1}{4 x^2}$.
Since $ \| G \| = \| x^{5/2} G \|_\infty$ and $\left | e^{x} y_1 
\right | \le Y_{1,M}$,  Lemma \ref{lem1prel}
applies with $w(x) = e^{-x} V_0 x^{-5/2}$ and $g(x)=x^{5/2} e^{x} G y_1$. Thus,
\begin{equation}
\label{17}
\Big | x^{5/2} \int_{-i \infty}^x V_0 (x') y_1 (x') G(x') dx' \Big |
\le \frac{Y_{1,M}}{14 \rho}  \| G \| 
\end{equation}
Since 
$e^{x} y_1$, $e^{-2x} z_2$ and $x^{5/2} G$ are bounded by $Y_{1,M}$, 
$z_{2,M}$ and $\|G \|$ respectively, 
we have
\begin{multline}
\label{19.0}
\Big | x^{5/2} e^{-x} \int_{-i \infty}^x  V_0 (x') 
y_1 (x') z_2 (x') G(x') \Big | \\
\le Y_{1,M} z_{2,M} \| G \| 
\left \{ |x|^{5/2} \int_{-i \infty}^x 
|e^{-x+x'}| \frac{1}{4 |x'|^{9/2}} |dx'| \right \} 
\le \frac{\sqrt{2}}{14 \rho} Y_{1,M} z_{2,M} \| G \|, 
\end{multline}
where we used the fact that on a vertical contour 
joining $-i \infty$ to $x \in \Omega_3$ we have
$|e^{x'-x} | = 1$ and $|dx'| \le \sqrt{2} d|x'|$.
 Lemma \ref{lem3prel} applied to $w(x) = e^{-x} x^{-5/2} V_0$, 
$g(x) = e^{x} y_1 x^{5/2} G$ gives
\begin{equation}
\label{18}
\Big | \frac{5 S^2}{24} x^{5/2} 
\int_{-i \infty}^x \log \frac{x'}{x}  V_0 (x') y_1 (x') G(x') \Big |
\le \frac{5 |S|^2}{288 \rho} \| G \| Y_{1,M} 
\end{equation}
Therefore, 
combining \eqref{17}-\eqref{18} and using 
(\ref{7.01}), 
we obtain
\begin{multline}
\label{19}
\Big | x^{3/2} \int_{-i \infty}^x \left ( y_2 (x) y_1 (x') - y_1 (x) y_2 (x') \right ) V_0 (x') G (x') \Big | 
\le Y_{1,M}^2 \left ( \frac{\sqrt{2}+1}{14 \rho} z_{2,M} 
+  
\frac{5 |S|^2}{288 \rho} \right ) \| G \| 
\end{multline}
Collecting the contributions of the terms involving $q$ and $V_0$  in \eqref{16} and \eqref{19} respectively,
it follows that 
\begin{multline}
\label{20}
\left\| \int_{-i \infty}^x 
\left ( y_2 (x) y_1 (x') - y_1 (x) y_2 (x') \right ) V(x')  
G(x') dx' \right\| \\
\le Y_{1,M} \left \{ \left ( 2 z_{2,M} M_q 
+ \frac{5 |S|^2}{24} M_{L,q} \right ) 
+ Y_{1,M} \left ( \frac{\sqrt{2}+1}{14 \rho} z_{2,M} 
+\frac{5 |S|^2}{288 \rho} \right ) \right \}
\| G \|  \\
=: V_M \| G \| 
\end{multline}
Since all the quantities involved in $V_M$ are decreasing in $\rho$,  it is clear that for $\rho\ge 3$,  $V_M$ is bounded by its value at $\rho=3$, which in turn is less than $9/40 $ and
Lemma \ref{lem2} follows.

\subsection{Nonlinear terms and proof of Lemma \ref{lem3}}
\label{lem3sub}

Applying Lemma \ref{lem1prel} to $w(x') = e^{-x'} {x'}^{-11/2} $,  
$g(x') = \tfrac12e^{x-x'} {x'}^{5} G^2 (x') e^{x'} y_1 (x')$, 
noting that
$|e^{x-x'} | \le 1$ on a radial contour in $\Omega_3$, and finally that
$|{x'}^5 G^2 (x')  | \le \| G \|^2$, $|e^{x'} y_1 (x') | \le Y_{1,M}$,
we obtain
\begin{equation}
\label{21}
\Big | e^{x} y_1(x) e^{-2x} z_2 (x)  x^{5/2} \int_{-i \infty}^x e^{x-x'}
\frac{e^{x'} y_1 (x') G^2 (x')}{2 \sqrt{x'}} dx' \Big |  
\le \frac{Y_{1,M}^2 z_{2,M}}{9 \rho^2} \| G \|^2
\end{equation}
Furthermore, we note that
\begin{equation}
\label{22}
\Big | e^{x} y_1(x) x^{5/2} \int_{-i \infty}^x e^{-x+x'}
\frac{e^{x'} y_1 (x') e^{-2x'} z_2 (x') G^2 (x')}{2 \sqrt{x'}} dx' \Big |  
\le \frac{\sqrt{2} Y_{1,M}^2 z_{2,M}}{9 \rho^2} \| G \|^2,
\end{equation}
where in \eqref{22} 
we have $|dx'| \le \sqrt{2} d |x'|$ 
and $|e^{x'-x} |=1$ on the vertical contour joining $-i \infty$ 
to $x \in \Omega_3$.
Applying Lemma \ref{lem3prel} to $w(x) = e^{-x} x^{-11/2} $,
$g(x) = 
{x}^{5} G^2 e^{x} y_1 (x)$, 
we get
\begin{equation}
\label{23}
\Big | \frac{5 S^2}{24} ~y_1(x) x^{5/2} \int_{-i \infty}^x  \log \frac{x'}{x} 
\frac{y_1 (x') G^2 (x')}{2 \sqrt{x'}} dx' \Big |
\le \frac{5 |S|^2 Y_{1,M}^2}{192 \rho^2} \| G \|^2
\end{equation}
Combining (\ref{21}), (\ref{22}) and (\ref{23}) we obtain
\begin{multline}
\label{24}
\left\| \int_{-i \infty}^x \left 
( y_2 (x) y_1 (x') - y_1 (x) y_2 (x') \right ) \frac{G^2 (x')}{2 \sqrt{x'}} dx'
\right\|  
\le \frac{Y_{1,M}^2}{\rho^2}  \left ( \left [ \frac{1+\sqrt{2}}{9} \right ] z_{2,M} 
+ \frac{5 |S|^2}{192} \right ) \|G \|^2 \\ 
=: T_M \| G \|^2 
\end{multline}
A very similar calculation shows that
$$ \left\| \int_{-i \infty}^x \left 
( y_2 (x) y_1 (x') - y_1 (x) y_2 (x') \right ) \frac{G_1^2 (x')-G_2^2 (x')}{2 \sqrt{x'}} dx'
\right\|  
\le 
T_M \left ( \| G_1\| +\| G_2 \| \right ) \| G_1 - G_2 \| 
$$

{\bf Proof of Lemma \ref{lem3}.} This now follows since a
calculation of $T_M$, which is a decreasing function of $\rho$,
shows $T_M \le \frac{18}{467} < \frac{1}{25}$ for 
$\rho \ge 3$.

\subsection{Bounds involving $G_0$ and proof of Lemma \ref{lem1}}
\label{lem1sub}

Using  (\ref{3.1}),\eqref{4.0} and the form of $y_2$ in (\ref{7.01}), 
it is convenient to decompose $G_0$ defined in \eqref{10} as follows
\begin{equation}
\label{8.1}
G_0  = G_{0,1} + G_{0,2} + G_{0,3} + G_{0,4} + G_{0,5} + G_{0,6} 
+G_{0,7}
\end{equation}
where
\begin{equation}
\label{8.2}
G_{0,1} (x) = y_1 (x) 
\int_{-i \infty}^x y_1 (x') \left [z_2 (x') - z_2 (x) \right ] 
~{\tilde R} (x')   
dx'
\end{equation}
\begin{equation}
\label{8.3}
G_{0,2} (x) = 
y_1 (x) \int_{-i \infty}^x y_{1,R} (x') \left [ z_2 (x') - z_2 (x) \right ] \left ( R_0 (x') 
+ R_1 (x') \right ) dx'
\end{equation}
\begin{equation}
\label{8.4} 
G_{0,3} (x) =
y_1 (x) \int_{-i \infty}^x \left [ y_{1,0} (x') + y_{1,1} (x') \right ] 
\left [ z_{2,R} (x') - z_{2,R} (x) \right ] \left ( R_0 (x') 
+ R_1 (x') \right ) dx'
\end{equation}
\begin{multline}
\label{8.5}
G_{0,4} (x) =  
y_1 (x) \int_{-i \infty}^x \left [ y_{1,0} (x') + y_{1,1} (x') \right ] 
\left \{ z_{2,0} (x') +z_{2,1} (x') \right . \\
\left. - z_{2,0} (x) -z_{2,1} (x) \right \} \left ( R_0 (x') 
+ R_1 (x') \right ) dx'
\end{multline}
\begin{equation}
\label{8.6.0}
G_{0,5} (x) 
= \frac{5 S^2}{24} y_1 (x) 
\int_{-i\infty}^x \log \frac{x'}{x} ~~y_1 (x') {\tilde R} (x') dx'   
\end{equation}
\begin{equation}
\label{8.6}
G_{0,6} (x) 
= \frac{5 S^2}{24} y_1 (x) 
\int_{-i\infty}^x \log \frac{x'}{x} ~~y_{1,R} (x') \left [ R_0 (x') + R_1 (x') \right ] ~dx'   
\end{equation}
\begin{equation}
\label{8.7} 
G_{0,7} (x) 
= \frac{5 S^2}{24} y_1 (x) 
\int_{-i \infty}^x \log \frac{x'}{x} \left ( y_{1,0} (x') + y_{1,1} (x') \right )
\left ( R_0 (x') + R_1 (x') \right ) dx'  
\end{equation}
In \S\ref{S9.1} below
we obtain bounds $M_j$ for $\|G_{0,j} \|$ for $j=1,2, \ldots, 7$; using them, we  get
\begin{equation}
\label{8.7.1}
\| G_0 \| \le \sum_{j=1}^7 M_j   
\end{equation}
The formulas for $M_j,j=1,\ldots, 7$ 
are given in the following subsections. These expressions will be shown to be decreasing in $\rho$, and Lemma  \ref{lem1} will follow using the values of $M_j$ at $\rho=3$.

\subsubsection{Bounds on $G_{0,1} $}\label{S9.1} 
Using Lemma \ref{lem1prel}, 
with $g(x') = e^{x-x'} e^{x'} y_1 (x')$ 
and $w(x') = e^{-x'} {\tilde R} (x')$, after d
eformation to a radial path of integration, we obtain
\begin{equation}
\label{8.8}
\Big | x^{5/2} y_1 (x) z_2 (x) \int_{-i\infty}^x y_1 (x') {\tilde R} (x') dx'  \Big |   
\le Y_{1,M}^2 z_{2,M}  \sum_{j=7}^9 \mathcal{F}_{1,j} \left [ r_j \right ] \rho^{-j/2+7/2} 
\end{equation}

Noting again that ${\tilde R} (x) + \frac{392}{625 x^{7/2}} $ has
only decaying exponentials, 
 Lemma \ref{lem1prel} (this time with $g(x) = e^{-x} y_1 (x) z_2 (x)$ and 
$w(x) = e^{x} \left ( {\tilde R} (x) + \frac{392}{625} x^{-7/2} \right )$ implies
\begin{equation}
\label{8.9}
\Big | x^{5/2} y_1 (x) \int_{-i\infty}^x y_1 (x') z_2 (x') \left [ {\tilde R} (x') + 
\frac{392}{625 {x'}^{7/2}} \right ] dx' \Big |   
\le Y_{1,M}^2 z_{2,M}   
\sum_{j=7}^9 \mathcal{F}_{1,j} \left [ {\tilde r}_j \right ] \rho^{-j/2+7/2} 
\end{equation}
Since on a vertical contour in joining $-i \infty$ to $ x \in \Omega_3$, 
$|e^{-x+x'} |=1$ and  $|dx'| \le \sqrt{2} d|x'|$, we get
\begin{equation}
\label{8.10}
\Big | x^{5/2} y_1 (x) \int_{-i\infty}^x y_1 (x') z_2 (x')  
\frac{392}{625 {x'}^{7/2}} dx' \Big |   
\le 
\frac{784}{3125}  
Y_{1,M}^2 z_{2,M} \sqrt{2}
\end{equation}
It follows that
\begin{multline}
\label{8.11}
\| G_{0,1} \| \le Y_{1,M}^2 z_{2,M} 
\left \{  \frac{784}{3125} \sqrt{2} + 
\mathcal{F}_{1,7} \left [ r_7 + {\tilde r}_7 \right ]  + 2 \sum_{j=8}^9 \mathcal{F}_{1,j} 
\left [ r_j \right ] \right \} 
=:
Y_{1,M}^2 z_{2,M} M_{G,1}
=:M_1
\end{multline} 
where the explicit expression  of $M_{G,1}$ is given in (\ref{A.16}) 
in the Appendix.

\subsubsection{Bounds on $G_{0,2}$.}
From \eqref{12.2}, \eqref{13.7} we note 
that $e^{2x} x ~y_{1,R}$ and $e^{-2x} ~z_2 (x)$ are  bounded by 
$Y_{1,R,M}$ and
$z_{2,M}$ respectively.  Lemma \ref{lem1prel} applied to  
$w(x)=
\frac{1}{x} \left ( R_0 + R_1 \right )$ implies
\begin{equation}
\label{8.12.0}
\| G_{0,2} \| \le 2 Y_{1,M} z_{2,M} Y_{1,R,M}   
\sum_{j=7}^8 \rho^{-j/2+7/2} \mathcal{F}_{1,j} 
\left [ r_{j-2} \right ]
=: 2 Y_{1,M} z_{2,M} Y_{1,R,M} M_{G,2} 
:= M_2
\end{equation}
where the formula for $M_{G,2}$ is given in (\ref{A.20}) in
the Appendix.

\subsubsection{Bounds on $G_{0,3}$:}

Note that $x e^{-x} z_{2,R}$ and $ e^{x} \left ( y_{1,0} + y_{1,1} \right )$ 
are bounded by
$z_{2,R,M}$ and $1 + \frac{|S|}{3 \sqrt{\rho}} $ resp. 
Lemma \ref{lem1prel} applied to 
$w(x') = \left [R_0 (x') + R_1 (x') \right ]/x' $ 
for the term containing $z_{2,R} (x')$ 
and to $w(x') = \left [R_0 (x') + R_1 (x') \right ]$ 
for the one containing $z_{2,R} (x)$ implies  
\begin{equation} 
\label{8.12}
\| G_{0, 3} \| 
\le Y_{1,M} z_{2,R,M} \left ( 1 + \frac{|S|}{3 \sqrt{\rho}} \right ) 
M_{G,3}
=:M_3
\end{equation}    
where 
\begin{equation}
\label{8.12.1}
M_{G,3} = M_{G,2}+ 
\sum_{j=5}^6 \rho^{-j/2+5/2} \mathcal{F}_{1,j} \left [ r_j \right ]. 
\end{equation}
The concrete expression of
$M_{G,3}$ is given in (\ref{A.20.1}) in the Appendix.
\subsubsection{Bounds on $G_{0,4} $:}
Using 
\eqref{4.0} and \eqref{5.01.0},
it follows that 
\begin{equation}
\label{8.13} 
T (x) =: \left ( y_{1,0}+y_{1,1} \right ) \left ( R_0 + R_1 \right )
= \sum_{j=5}^7 x^{-j/2} t_{j} (e^{x}),   
\end{equation}
where $t_j (\zeta)$ are polynomials in $1/\zeta$, having no constant and linear terms; the precise expressions are 
in \eqref{8.13.1}-\eqref{8.13.3} in the Appendix.
\begin{equation}
\label{8.14}
U (x) =: \left ( y_{1,0}+y_{1,1} \right ) \left ( z_{2,0} + z_{2,1} \right ) \left ( R_0 + R_1 \right )
=\sum_{j=5}^8 x^{-j/2} u_j (e^{x}) 
\end{equation}
where $u_j(\zeta)$ are polynomials in $1/\zeta$ without
constant terms--see
\eqref{8.14.1}-\eqref{8.14.4} in the Appendix. 
We also note that
\begin{equation}
\label{8.15} 
T(x) = \frac{d}{dx} \left [ \sum_{j=5}^7 ~x^{-j/2} ~\tau_j \left (e^{x} \right )\right ] 
+ \sum_{j=7}^9 ~x^{-j/2} ~{\tilde t}_j (e^{x} ) , 
\end{equation}
where $\tau_j (\zeta) $, 
${\tilde t}_j (\zeta)$,  are polynomials in $1/\zeta$,
see  \eqref{8.15.1}-\eqref{8.15.6}  in the Appendix, with no constant or linear terms.
Again, we note that
\begin{equation}
\label{8.16}
U(x) = \frac{d}{dx} \left [ \sum_{j=5}^8 x^{-j/2} \nu_j \left ( e^{x} \right ) \right ]
+ \sum_{j=7}^{10} x^{-j/2} {\tilde u}_j \left ( e^{x} \right ) 
\end{equation}
where the polynomials  in $1/\zeta$, ${\tilde u}_j (\zeta)$,
$\nu_j (\zeta)$ have no constant terms, see 
\eqref{8.16.1}-\eqref{8.16.8}.
Using \eqref{8.5}, \eqref{8.13}-\eqref{8.16},
it follows that 
\begin{multline}
\label{8.17}
G_{0,4} (x) = y_1 (x) \left [ 
\sum_{j=5}^8 x^{-j/2} \nu_j \left (e^{x} \right ) - \left (z_{2,0} (x) 
+ z_{2,1} (x) \right ) \sum_{j=5}^7 x^{-j/2} \tau_j \left (e^{x} \right ) \right ] \\
+ y_1 (x) \int_{-i\infty}^x 
\left [ \sum_{j=7}^{10} {x'}^{-j/2} {\tilde u}_j \left (e^{x'} \right ) - 
\left [ z_{2,0} (x) + z_{2,1} (x) \right ] \sum_{j=7}^9 {x'}^{-j/2} {\tilde t}_j \left ( e^{x'} \right ) 
\right ] dx'      
\end{multline}
Straightforward calculations show that 
\begin{equation}
\label{8.18}
\sum_{j=5}^8 x^{-j/2} \nu_j \left (e^{x} \right ) - 
\left (z_{2,0} (x) +z_{2,1} (x)  \right ) \sum_{j=5}^7 x^{-j/2} \tau_j \left (e^{x} \right )   
=
\sum_{j=5}^8 x^{-j/2} p_j \left ( e^{x} \right ),
\end{equation}    
where  the $p_j(\zeta)$s are also polynomials in $1/\zeta$ with no constant term; see the Appendix, starting with eq. \eqref{8.18.1}.

Applying Lemma \ref{lem1prel} 
to the  two terms in the integral on the right of    
(\ref{8.17}), using $|e^{x-x'} |\le 1$ and 
$w(x) = \sum_{j=7}^{10} x^{-j/2} {\tilde u}_j \left (e^{x} \right ) $ 
for the first integral and
$w(x) = \sum_{j=7}^{9} x^{-j/2} {\tilde t}_j \left (e^{x} \right ) $ 
for the second, we obtain
\begin{multline}
\label{8.19}
\| G_{0, 4} \| \le Y_{1,M} \left [ \sum_{j=5}^8 \rho^{-j/2+5/2} {\tilde p}_j (1) \right ] 
+ Y_{1,M} \left ( \left ( \frac{1}{2} + \frac{2 |S|}{3 \sqrt{\rho}} \right ) 
\sum_{j=7}^9 \rho^{-j/7+7/2} \mathcal{F}_{1,j} \left [ {\tilde t}_j 
\right ] \right .\\
\left . + \sum_{j=7}^{10} \rho^{-j/7+7/2} \mathcal{F}_{1,j} \left [ {\tilde u}_j \right ]
\right ) 
=: Y_{1,M} \left ( M_{G,4,0}+M_{G,4,1} \right ) =: M_4,
\end{multline}
see  (\ref{A.21}), (\ref{A.22}), where 
${\tilde p}_j $ is the
polynomial obtained from $p_j $ by 
replacing each of its coefficients by its absolute value.

\subsubsection{Bounds on $G_{0,5} $, $G_{0,6} $ and $G_{0,7}$}
For $G_{0,5}$ we simply use Lemma \ref{lem3prel} with $w(x) = e^{-x} {\tilde R} (x)$,
and $g(x) = e^{x} y_1$.   
We obtain 
\begin{multline}
\label{9.1.1}
\Big | x^{5/2} \frac{5 S^2}{24} y_1 \int_{-i \infty}^x 
\log \frac{x'}{x} y_1 (x') {\tilde R} (x') \Big |
\le \frac{5 |S|^2}{24} Y_{1,M}^2 \sum_{j=7}^9  |x|^{-j/2+7/2}  \mathcal{F}_{3,j} 
\left [ r_j \right ] 
=: \frac{5 |S|^2}{24} Y_{1,M}^2 M_{G,5} =: M_5  
\end{multline}
where  $M_{G,5}$ is given in (\ref{A.23}) in the Appendix.
We again use Lemma \ref{lem3prel} with $w(x) = 
\frac{e^{-x}}{x} \left ( R_0 + R_1 
\right )$ and $g(x) = e^{x} x y_{1,R} $, to obtain 
\begin{multline}
\label{9.2}
\Big | x^{5/2} \frac{5 S^2}{24} y_1 \int_{-i \infty}^x 
\log \frac{x'}{x} y_{1,R} (x') \left ( R_0 (x') + R_1 (x') \right ) \Big |
\\\le \frac{5 |S|^2}{24} Y_{1,M} Y_{1,R,M} 
\sum_{j=7}^9 \rho^{-j/2+7/2} \mathcal{F}_{3,j} \left [ r_{j-2} \right ]  
=:
\frac{5 |S|^2}{24} Y_{1,M} Y_{1,R,M} M_{G,6} =: M_6,
\end{multline}
where $M_{G,6}$ is given in (\ref{A.24}) in the Appendix.
Now consider $G_{0,7} (x) $.
Recall that 
\begin{equation}
\label{9.3} 
T(x) = \left ( y_{1,0} + y_{1,1} \right ) \left ( R_0 + R_1 \right )
= \sum_{j=5}^7 x^{-j/2} t_j \left ( e^{x} \right ) 
\end{equation}
From the expressions of $t_j$ in (\ref{8.13.1})- (\ref{8.13.3}), 
it is clear that 
(\ref{9.3}) involves only
decaying  exponentials. Therefore, we may apply Lemma \ref{lem4prel} to 
$w(x)=T(x)$ giving
\begin{multline}
\label{9.4}
\Big | x^{5/2} \frac{5 S^2}{24} y_1  (x) \int_{-i \infty}^x 
\log \frac{x'}{x} \left ( y_{1,0} (x') + y_{1,1} (x') \right ) 
\left ( R_0 (x') + R_1 (x') \right ) dx' \Big | \\
\le \frac{5 |S|^2}{24} Y_{1,M} \sum_{j=5}^7 \rho^{-j/2+5/2} \mathcal{F}_{4,j} 
\left [ t_j \right ] =:
\frac{5 |S|^2}{24} Y_{1,M} M_{G,7} =: M_7
\end{multline}
where $M_{G,7}$ is given in (\ref{A.25}) in the Appendix.

\section{Analysis of $y_t$ for $|z| \le 1.7$}
For the analysis of the inner region $|z|<17/10$ it is convenient to use the symmetry of the equation and write it as
\begin{equation}
  \label{eq:P1g}
  g''=6g^2+t \ \text{where}\ g(t)=e^{2\pi i/5}y(-te^{\pi i/5})
\end{equation}
We take initial conditions
close to  \eqref{eqy0z0}, \eqref{eqy0z0p}, converted into conditions for $g(t)$
\begin{equation}
  \label{eq:ic1}
  g(t_0)=-\frac{280}{519};\ \ g'(t_0)=\frac{150}{1013};\ \text{where } t_0=-1.7
\end{equation}
We first construct a polynomial which is sufficiently close  to $g$ in $L^\infty([t_0,0])$ so as to be able  to preserve error bounds of the order of those in \eqref{erry0z0}. The way the polynomial $g_0$ below was obtained is essentially  by projecting $g$, calculated from its truncated Taylor series, on Chebyshev polynomials of order seven, 
enough for the aforementioned goal;  the polynomial is
\begin{equation}
  \label{eq:defgt}
g_0(t)=-\frac {280}{519}+\frac {150s}{1013}+\frac {239s^2}{10331}+\frac {110s^3}{14779}-\frac {32s^4}{9853}+\frac {9s^5
}{4397}-\frac {16s^6}{39505}+\frac {8s^7}{49105},\ \text{where}\ s=t-t_0.
\end{equation}
\begin{Definition}\label{D1}
 {\rm For a polynomial $P(x)=\sum_{k=0}^n c_k
   x^k$  on an interval $I$   we define an $\ell^1$ norm by  $\|P\|_1=\sum_{k=0}^n|c_k|m^k$ where $m=\sup_{I}|x|$. }
\end{Definition}
Taking $g=g_0+\delta$ we get
\begin{equation}
  \label{eq:eqdel}
  \delta''-12g_0\delta=6\delta^2+R;\ \delta(0)=a_1=g(0)-g_0(0),\delta'(0)=a_2=g'(0)-g'_0(0)
\end{equation}
where $-R=g_0''-6g_0^2-t$ is a polynomial of degree 14. 

\begin{Proposition}\label{Pp16}
For $t\in[t_0,0]$ we have $|R|<1/{8619}$.
\end{Proposition}
\begin{Note}\label{N1} (a) {\rm  Estimating rigorously and with good
    accuracy, a polynomial $P(x)$ on an interval $I=[a,b]$ is
    elementary, and it can be done efficiently in a
    number of ways.  The one used here is the same as in
    \cite{NLS}. We choose a suitable partition of $[a,b]$, 
    $\pa=\langle x_0, x_1,...,x_{n-1}, x_n\rangle$, 
where $x_0=a$, $x_n=b$; then
      write $t=\tfrac12(x_{i}+x_{i-1})+u$ on each subinterval
      $[x_{i-1},x_{i}]$ for $i=1,..n$ and  re-expand $P$ to
    obtain a polynomial in $u$. The polynomial in $u$ is estimated
    by taking the extremum of the modulus of the cubic subpolynomial and placing 
    $\| \cdot \|_1 $ on the rest. In practice a small
    number of partition points suffices to get a good bound. 

Since all partition points are nonpositive, to simplify the writing we denote
 by $-\pa$ the partition $\langle -x_0,-x_1,...,-x_n\rangle$.

    (b) Bounding rational functions with real coefficients reduces to
    estimating polynomials since the inequality $|P/Q|<\ve$
    with $Q>0$ is equivalent to the pair of inequalities
    $P-\ve Q<0$ and $P+\ve Q>0$. However, in our case the denominator is
      $W$ which is very close to one, 
 so an upper bound of the modulus of the numerator and a lower bound of $W$ give an equivalent accuracy.
}
\end{Note}
\begin{proof}\
Use Note \ref{N1} (a) and the partition
  $\pa= -\langle \frac{17}{10}, {\frac {27}{20}},\frac45,\frac{3}{10},{\frac {3}{20}}, 0 \rangle$.
  \end{proof}
     To write the equation for $\delta=g-g_0$ in an integral form suitable
    for a contraction argument, we would need the fundamental solution
    of the linear operator on the left side of \eqref{eq:eqdel}. Of
    course this cannot be done in closed form; once more, we use a
    pair of polynomials close enough in $L^\infty$ to a fundamental
    system and estimate the errors introduced in this way.  The
    pair of polynomials is obtained in the same way as $g_0$ was found,
    and is given by 
  \begin{multline}
    \label{eq:eqJ12}
    J_1(t)=1-{\frac {9489\ds s^{2}}{2932}}+{\frac {1350\ds s^{3}}{4721}}+{\frac {
359\ds s^{4}}{199}}-{\frac {1526\ds s^{5}}{3719}}-{\frac {708\ds s^{6}}{1633}}
+{\frac {503\ds s^{7}}{2201}}-{\frac {211\ds s^{8}}{6486}}\ (s=t-t_0)
\\J_2(t)=s-{\frac {48\ds s^{2}}{659797}}-{\frac {2941\ds s^{3}}{2730}}+{\frac {
675\ds s^{4}}{4873}}+{\frac {1832\ds s^{5}}{4745}}-{\frac {2305\ds s^{6}}{19401
}}-{\frac {677\ds s^{7}}{14054}}+{\frac {1573\ds s^{
8}}{53783}}-{\frac {531\ds s^{9}}{128216}}\ \ (s=t-t_0)
  \end{multline}
Let 
$W=J_1J_2'-J_2J_1'$ be their Wronskian. 
The equation that the pair $(J_1,J_2)$ satisfies is
\begin{equation}
  \label{eq:J12e}
  f''+Af'+Bf=0
\end{equation} 
where \begin{equation}
  A:=\frac{J_2J_1''-J_1J_2''}{W}, \ B:=\frac{J_2''J'_1-J_1''J'_2}{W}
\end{equation}
The Green's function associated to \eqref{eq:J12e} is 
\begin{equation}
  \label{eq:Gr1}
  \mathcal{G}(s,t)=W(s)^{-1}[J_2(t)J_1(s)-J_1(t)J_2(s)]
\end{equation}
We define the linear operators $K_1$, $K_2$ such that
\begin{equation}
  \label{eq:defk12}
  K_1 [f] (t) =\int_{t_0}^t\mathcal{G}(s,t)f(s)ds,\ K_2 [f] (t)
=\int_{t_0}^t\frac{\partial \mathcal{G}(s,t)}{\partial t}f(s)ds,\ 
\end{equation}
We can now rewrite now  \eqref{eq:eqdel} as an equivalent integral system.

 Let 
$B_1=12g_0+B$, $\boldsymbol{\delta}=(\delta,\delta')$ and
$r(\boldsymbol{\delta}(s),s)=A\delta'+B_1\delta +6\delta^2$,
$\mathbf{J}_1=(J_1,J'_1)$,  $\mathbf{J}_2=(J_2,J'_2)$ and the vector operator
$\mathbf{K}$ defined by $\mathbf{K} [f] = \left ( K_1 [f], K_2 [f] \right )$.
We have
\begin{equation}
  \label{eq:syst3}
  \boldsymbol{\delta}=a_1\mathbf{J}_1+ a_2\mathbf{J}_2-\mathbf{K} R 
+\mathbf{K} r=\hat{N} \boldsymbol{\delta}
\end{equation}
and \eqref{erry0z0} implies
\begin{equation}
  \label{eq:vala12}
 |a_1|< \alpha_1:=1/290 \ \text{and}\  |a_2|<\alpha_2:=1/152
\end{equation}
To analyze
the integral system  we first estimate the various quantities 
in \eqref{eq:syst3}.
\begin{Proposition}\label{Pp17}
  The following estimates hold in the sup norm on $[t_0,0]$
  \begin{equation}
    \label{eq:estJW}
   \max\{\|J_1\|, \|J_1/W\|\}\le \frac{6}{5};\  \max\{\|J_2\|,\|J_2/W\|\} \le
   \frac3{7}, \ \|J_1'\|\le \frac{5}{2}, \  \|J'_2\| \le \frac{21}{20}
  \end{equation}
  \begin{equation}
    \label{eq:est2c}
 |W-1|<1/{500}, \   \|A\|<1/{1216},\ \|B_1\|<1/{492};\
  \end{equation}
\begin{equation}
    \label{eq:j1j2}
   \sup_{|a_1|<\alpha_1,|a_2|<\alpha_2} \| a_1J_1+a_2J_2\|<1/{180};\
\sup_{|a_1|<\alpha_1,|a_2|<\alpha_2} \| a_1J'_1+a_2J'_2\|<1/{90}\
  \end{equation}
\end{Proposition}\label{Pes}
\begin{proof}
  The proofs are based on Note \ref{N1} (a) to estimate the
  polynomials and rational functions. 

We illustrate the calculation on a rational function, $B_1$ on a sample
interval, say $(-\frac32,-\frac{11}{10})$, see below. We thus take $s=-13/10+
u/5$  and simplify the resulting expression. Denoting by $E_j$ polynomials  with $\ell^1$ norm less than
$1/1000$ we simply get on this interval
$$B_1=\left(\frac {3}{2332}-\frac{2u}{60137}-\frac{22u^2}{3361}+\frac{6u^3}{7241}+E_1\right)/(1+E_2)$$
The derivative of the cubic has one root for $u=-[1,1]$; the value of the
cubic there is $<1.3\cdot 10^{-2}$.  Checking it at the endpoints
of the interval as well, we see that this is its maximum absolute value. The
other calculations are as straightforward as this, so we naturally  omit the details.
 
We found it convenient to use a different partition $\pa$ 
for estimating maximal absolute values of each function. 
We chose partitions $-\langle \frac{17}{10}, \frac{1}{2}, 0 \rangle$ and
$-\langle  \frac{17}{10}, \frac{11}{10}, \frac{7}{10}, 0 \rangle$
for $J_1$ and $J_1^\prime$ respectively (\emph{cf}. again Note \ref{N1}, (b));
$ -\langle  \frac{17}{10}, \frac{11}{10}, \frac{1}{2}, 0  \rangle$,
$-\langle \frac{17}{10}, \frac{11}{10}, 0 \rangle$ 
and $-\langle \frac{17}{10}, \frac{17}{20}, 0 \rangle$ 
for $J_2$, $J_2^\prime$ and $W$ respectively. 
  Finally, we use
  Note \ref{N1} (b) to estimate $A$ and $B_1$ using the partitions
  $-\langle\frac{17}{10}, \frac{13}{10}, 1, \frac{7}{10}, \frac{3}{10},
  \frac{1}{10}, 0 \rangle $
  and $-\langle\frac{17}{10}, \frac{3}{2}, \frac{11}{10}, \frac{7}{10}, \frac{2}{5},
  \frac{1}{4}, \frac{1}{10}, 0 \rangle $ respectively.

  For \eqref{eq:j1j2} we note that $f(s, {\bf a}):=a_1 J_1 (s)+a_2 J_2 (s)$ and
  $f'(s; {\bf a})$ are linear in $(a_1,a_2)$ and thus for any $s$, the
  extrema of $|f(s; {\bf a})|$, $|f'(s, {\bf a})|$ are reached on the boundary. Thus we
  only need bounds when $a_1=\alpha_1$, $a_2=\pm \alpha_2$. The
  partition points chosen are  $=-\langle\frac{17}{10}, \frac{1}{2}, 0
    \rangle$ for $f(s,{\boldsymbol{\alpha}})$, 
  $-\langle\frac{17}{10}, \frac{4}{5}, 0 \rangle$ for $f(s;
  \alpha_1,-\alpha_2)$, 
  $-\langle\frac{17}{10}, \frac{9}{10}, 0 \rangle$ for $f'(s;
  {\boldsymbol{\alpha}})$ and 
  $-\langle\frac{17}{10}, \frac{7}{5}, \frac{11}{10}, 0 \rangle$ for $f'(s; \alpha_1,-\alpha_2))$.
\end{proof}

To analyze \eqref{eq:syst3} we use  the norm
$\|\boldsymbol{\delta}\|^{(\frac12)}=\max\{\|\delta\|,\tfrac12\|\delta'\|\}$ 
where $\|f\|$ is the
sup norm on $[t_0,0]$.
\begin{Proposition}\label{Pes1}
(i) The nonlinear operator $\hat{N}$ in \eqref{eq:syst3} is contractive in the
disc $\{\boldsymbol{\delta}:\|\boldsymbol{\delta}\|^{(\frac12)}
\leqslant1/158\}$;
the contractivity factor is
$<1/6$.

(ii) We have $\ve_1=\|\delta-a_1J_1-a_2 J_2\|<1/1500$,
$\ve_2:=\|\delta'-a_1J'_1-a_2 J'_2\|<1/658.$

(iii)  Let $a={87}/{469}$ and $b=41/134$. We have
\begin{equation}
    \label{eq:22}
   |g(0)+a|<1/{167};\ 
 |g'(0)-b|<\ve=1/{108} 
  \end{equation}
\end{Proposition}
\begin{Note}
The values of  $a$ and $b$ are rational approximations of $g_0 $ and $g_0^\prime$ at $t=0$ (see \eqref{eq:defgt}). They are, as expected,  
consistent with the ones numerically calculated in  \cite{Nalini}.
\end{Note}
The proof of (i) is a simple calculation based on the estimates in Propositions \eqref{Pp16} and \eqref{Pp17}.
The
  integrals are estimated crudely, by placing an absolute value on all
  terms and multiplying with the length of the interval, $17/10.$ 

For (ii) we note that
 $$\ve_1\leqslant
 \|K_1\|\(2\|A\|\|\boldsymbol{\delta}\|^{(\frac12)}+\|B_1\|\boldsymbol{\delta}+6\|\boldsymbol{\delta}\|^2+\|R\|\)$$
(iii) At $t=0$ ($s=1.7$) we have
$$|g(0)+a|\leqslant
|g_0(0)+a|+\max_{|a_j|<\alpha_j,j=1,2}|a_1J_1(0)+a_2J_2(0)|+\epsilon_1<1/167$$
where we used  (ii),  \eqref{eq:defgt}, \eqref{eq:eqJ12}  and  \eqref{eq:vala12}; $g'(0)$ is estimated in a very similar way.

\section{The Maclaurin series of $y_t$} 
\begin{Proposition}\label{P20a}
  The  Maclaurin series of the function $t\mapsto g(t)$ converges in a disk of radius at least $R_0=1.85$.
\end{Proposition}
\begin{proof}
  We define  $c_i, i=0,1,2,...$ to be the Taylor coefficients of $g$ at
$t=0$ and  note that
\begin{equation}
  \label{eq:coeffs}
  c_2=3c_0^2>0
\end{equation}
The recurrence relation for the Taylor coefficients of the solution of
\eqref{eq:P1g} with initial condition $g(0)=c_0=-a,g'(0)=c_1=b$ is
\begin{equation}
  \label{eq:recgen}
  (k+1)(k+2)c_{k+2}=6\sum_{j=0}^k c_j c_{k-j};\ k>1
\end{equation}
We now take the full range of initial  conditions compatible with the error
range  \eqref{eq:22} (for simplicity we take the larger of the two bounds, 
$\epsilon$):  $\tilde{c}_0=-a+\ve \sigma_1$, $\tilde{c_1}=b+\ve
\sigma_2$ where $\sigma_1, \sigma_2\in I= [-1,1]$, and calculate the formal
series solution at zero for these initial data; we denote by  $\tilde{c}_i$
the  Taylor coefficients thus calculated; we have
\begin{equation}
  \label{eq:newvals}
  \tilde{c}_0=-a+\ve \sigma_1,\,\tilde{c}_1=b+\ve \sigma_2,\tilde{c}_2=
3(-a+\ve \sigma_1)^2,\,\tilde{c}_3=2(-a+\ve \sigma_1)(b+\ve \sigma_2)+\tfrac16
\end{equation}
The coefficients  
$-\tilde{c}_0,\tilde{c_1},c_2,\tilde{c}_3$ in
\eqref{eq:newvals} can clearly be 
maximized/minimized by elementary means 
as functions of $(\sigma_1,\sigma_2)\in I^2$. The result is
\begin{equation}
  \label{eq:bds}
 0<-\tilde{c}_0<1/5, 0<\tilde{c}_1<6/19, 0< \tilde{c}_2<1/{8},\ 0<\tilde{c}_3<1/{15}
\end{equation}
We write
\begin{equation}
  \label{eq:ck}
(k+2)(k+1)  |\tilde{c}_{k+2}|<6\sum_{j=0}^{k}|\tilde{c}_j||\tilde{c}_{k-j}|;\ k\ge 2
\end{equation}
where $\tilde{c}_i,i=0,...,3
$ are taken to be the upper bounds in \eqref{eq:bds}.
We check that for $k=0,1,2,3$ we have
\begin{equation}
  \label{eq:est012}
  |\tilde{c}_k|<(k+1)/R_0^{k+2};\ \ \text{ where } R_0=\frac{37}{20}=1.85
\end{equation}
\begin{Lemma}
The inequality \eqref{eq:est012} is satisfied for all $k\ge 0$.
 \end{Lemma}
\begin{proof}
  We note that for any $\rho>0$ the sequence $a_k=(k+1)\rho^{k+2},k=0,1,2,...$ is an exact
  solution of the recurrence
  \begin{equation}
    \label{eq:eqtest}
     (k+1)(k+2)a_{k+2}=6\sum_{j=0}^k a_j a_{k-j}, \forall \,\,k\geqslant 0
  \end{equation}
The rest is straightforward induction.
\end{proof}
Proposition \ref{P20a} now follows from \eqref{eq:est012}.
\end{proof}

\section{End of proof of Theorem \ref{Thm1}}

We have already shown that in each domain $\Omega_I$, $\Omega_1 \cup
\Omega_2$ the unique solution we obtained equals the tritronqu\'ee
$h_t (x)$.  After changes of variables, by matching at $z=z_0 = 1.7
e^{i \pi/5} (\text{\emph{i.e.,} at\,\,} t=t_0=-1.7)$, we then proved
that $y_t(z)$ has a convergent Maclaurin series with radius of
convergence $\ge 1.85$. Since the solution is by construction
  continuous in the closed domains $\overline{\Omega_1 \cup \Omega_2}$
  and $\overline{\Omega_3}$, the integral reformulation implies that
  its derivative is also continuous in these closed regions. By
  standard results on local existence and analyticity, the solution is
  then analytic in a neighborhood of any point on the boundary.
Therefore, it follows that $y_t (z)$ is analytic in the domain $
\left \{ -\frac{3}{5 \pi } \le \arg z \le \frac{\pi}{5} , |z| \ge 1.7
\right \} \cup \left \{z: |z| \le 1.85 \right \}$.  By the symmetry of
the solution, see Remark \ref{SymR} and the Schwartz reflection
principle, regularity also follows for $\arg z\in[\pi/5,\pi]$.

In the process, we determined  $y_t (0)$ and $y_t^\prime (0)$ to within
$<10^{-2}$ proven  error bounds  (see \eqref{eq:22} and \eqref{eq:P1g})
in agreement with \cite{Nalini}.

\section{Appendix: The concrete expressions of various terms}
\small\small 
\begin{equation}
\label{3.1.1}
r_5 (\zeta) = -{\frac {53}{64}}\,{\frac {{S}^{2}}{{\zeta}^{2}}}+{\frac {161}{1728}}
\,{\frac {{S}^{4}}{{\zeta}^{4}}}-{\frac {35}{41472}}\,{\frac {{S}^{6}}
{{\zeta}^{6}}} \\,
~r_6 (\zeta) = 
-{\frac {995}{2304}
}\,{\frac {{S}^{3}}{{\zeta}^{3}}}+
{\frac {301}{20736}}\,{\frac {{S}^{5}}{{\zeta}^{5}}}
-{\frac {11}{124416}}\,{\frac {{S}^{7}}{{\zeta}^{7}}}
\end{equation}
\begin{equation}
\label{3.1.3}
r_7 (\zeta) = -{\frac {392}{625}}
-{\frac {5551}{9216}}\,{\frac {{S}
^{2}}{{\zeta}^{2}}}
-{\frac {1417}
{165888}}\,{\frac {{S}^{4}}{{\zeta}^{4}}}
+{\frac {289}{248832}}\,{\frac {{S}^{6}}{{\zeta}^{6}}}
-{\frac {23}{2985984}}\,{\frac {{S}^{8}}{{\zeta}^{8
}}}
\end{equation}
\begin{equation}
\label{3.1.4}
r_8 (\zeta) = 
\frac{225 S}{512 \zeta} 
-{\frac {2051}{9216}}\,{\frac {{S}^{3
}}{{\zeta}^{3}}} 
-{\frac {241}{55296}
}\,{\frac {{S}^{5}}{{\zeta}^{5}}}
+{\frac {23}{186624}}\,{\frac {{S}^{7}}{{\zeta}^{7}}}
-{\frac {5}{8957952}}\,{\frac {{S}^{9}}{{\zeta}^{9}}}
\end{equation}
\begin{equation}
\label{3.1.5}
r_9 (\zeta) = 
-{\frac {81}{32768}}\,{\frac {{S}^{2}}{{\zeta}^{2}}}
+{\frac {43}{16384}}\,{\frac {{S}^{4}}{{\zeta}^{4}}}
-{\frac {947}{1327104}}\,{\frac {{S}^{6}}{{\zeta}^{6}}}
+{\frac {215}{
23887872}}\,{\frac {{S}^{8}}{{\zeta}^{8}}}
-{\frac {25}{859963392}}\,{
\frac {{S}^{10}}{{\zeta}^{10}}}
\end{equation}

\begin{equation}
\label{7.1}
q_5 (\zeta) = -{\frac {539}{384}}\,{\frac{S}{{\zeta}^{2}}}+{\frac {307}{864}}\,{\frac {{S}^{3}}{{
\zeta}^{4}}}-{\frac {35}{6912}}\,{\frac {{S}^{5}}{{\zeta}^{6}}} \\,~
q_6 (\zeta) = 
-{\frac {1361}{1152}
}\,{\frac {{S}^{2}}{{\zeta}^{3}}}
+{\frac {727}{10368}}\,{\frac {{S}^{4}}{{\zeta}^{5}}}
-{\frac {77}{124416}}\,{\frac {{S}^{6
}}{{\zeta}^{7}}}
\end{equation}
\begin{equation}
\label{7.3}
q_7 (\zeta) = -{\frac {95}{96}}\,{\frac {S}{{\zeta}^{2}}}
-{\frac {3817}{165888}}\,{\frac {{S}^{3}}{{\zeta}^{4}}}
+{\frac {277}{41472}}\,{
\frac {{S}^{5}}{{\zeta}^{6}}}
-{\frac {23}{373248}}\,{\frac {{S}^{7}}{{
\zeta}^{8}}}
\end{equation}
\begin{equation}
\label{7.4}
q_8 (\zeta) = -{\frac {621 {S}^{2}}{1024 {\zeta}^{3}}}
-{\frac {1591}{82944}}\,{\frac {{S
}^{4}}{{\zeta}^{5}}}
+{\frac {623}{746496}}\,{\frac {{S}^{6}}{{\zeta}^{7}}}
-{\frac {5}{995328}}\,{\frac {{S}^{8}}{{\zeta}^{9}
}}
\end{equation}
\begin{equation}
\label{7.5}
q_9 (\zeta) = \frac{15}{2048} \frac{S^3}{\zeta^4}\, 
- {\frac {3515}{884736}}\,{\frac {{S}^{5}}{{\zeta}^{6}}}
+{\frac {1675}{23887872}}
\,{\frac {{S}^{7}}{{\zeta}^{8}}}
-{\frac {125}{
429981696}}\,{\frac {{S}^{9}}{{\zeta}^{10}}}
\end{equation}

\begin{equation}
\label{13.1}
E_5 (\zeta) = 
-{\frac {269}{576}}\,{\frac {{S}^{3}}{\zeta}}+{\frac {61}{10368}}\,{
\frac {{S}^{5}}{{\zeta}^{3}}} \\, ~
E_6 (\zeta) = 
-{\frac {1691}{20736}}\,{\frac {{S}^{4}}{{\zeta}^{2}}}+{\frac {353}{
497664}}\,{\frac {{S}^{6}}{{\zeta}^{4}}}
\end{equation}
\begin{equation}
\label{13.4}
E_7 (\zeta) = 
{\frac {95}{576}}\,{\frac {{S}^{3}}{\zeta}}-{\frac {1915}{248832}}\,{
\frac {{S}^{5}}{{\zeta}^{3}}}+{\frac {25}{373248}}\,{\frac {{S}^{7}}{{
\zeta}^{5}}} \\, ~
E_8 (\zeta) =
{\frac {25}{768}}\,{\frac {{S}^{4}}{{\zeta}^{2}}}-{\frac {125}{165888}
}\,{\frac {{S}^{6}}{{\zeta}^{4}}}+{\frac {625}{143327232}}\,{\frac {{S
}^{8}}{{\zeta}^{6}}}
\end{equation}

\begin{multline}
\label{A.1.0}
E_M (\rho) = 
\left( {\frac {269 |S|^3}{288}}\, +{
\frac {61 |S|^5}{15552}}\, \right) 
\rho^{-3/2} 
+ \left( {\frac {1691 |S|^4}{20736}}\,
+{\frac {353 |S|^6}{995328}}\,
\right) {\rho}^{-2} \\
+ \left( {
\frac {95 |S|^3}{288}} +{\frac {1915 |S|^5
}{373248}}+{\frac {5 |S|^7}{186624}
} \right) \rho^{-5/2}  
+ \left( {\frac {25 |S|^4}{768}}\,  
+ {\frac {125 |S|^6}{331776}}\,  
+{\frac {625 |S|^8}{429981696}}\,  
\right) {\rho}^{-3}
\end{multline}
\begin{equation}
\label{A.1.2}
M_q = \sum_{j=3}^7 M_{q,j}  \rho^{-j/2}, 
\end{equation}
where
\begin{equation}
\label{A.1.3}
M_{q,3} =\frac{539}{1536} \left| S \right| +{\frac {307}{3456}}\, \left| S
 \right|^{3}+{\frac {35}{27648}}\, \left| S \right|^5  \\,~
M_{q,4} = 
{\frac {
1361}{5184}}\, \left | S \right|^{2}
+{\frac {727}{46656}}\, \left| S \right|^{4}
+{\frac {77}{
559872}}\, \left| S \right|^{6}
\end{equation}
\begin{equation}
\label{A.1.5}
M_{q,5} =  
{\frac {19}{96
}}\, \left| S \right| 
+{\frac {3817}{829440}}\, \left| S
 \right|^{3}
+{\frac 
{277}{207360}}\, \left| S \right|^{5}
+{\frac {23}{1866240}}\, \left| S \right|^{7}
\end{equation}
\begin{equation}
\label{A.1.6}
M_{q,6} = 
{\frac {621}{5632}}\, \left| S \right|^{2}+{\frac {1591}
{456192}}\, \left| S \right|^{4}
+{\frac {623}{4105728}}\,
\left| S \right|^{6}
+{\frac {5}{5474304}}
\left| S \right|^{8}
\end{equation}
\begin{equation}
\label{A.1.7}
M_{q,7} = 
\frac{5}{4096} \left | S \right |^3 + 
{\frac {3515}{5308416}}\, \left| S \right|^{5}
+{\frac {
1675}{143327232}}\, \left| S \right|^{7}
+{\frac 
{125}{2579890176}}\, \left| S \right|^{9}
\end{equation}
\begin{equation}
\label{A.1.8}
M_{L,q} = \sum_{j=3}^{7} M_{L,q,j} \rho^{-j/2}, 
\end{equation}
where
\begin{equation}
\label{A.1.9}
M_{L,q,3} =  
{\frac {77}{192}}\, \left| S \right| +{\frac {307}{3024}}\,  
 \left| S \right|^{3}+{\frac {5}{3456}}\, \left| S
 \right|^{5} \\, ~
M_{L,q,4} =  
{\frac {
1361}{4608}}\, \left| S \right|^{2}
+{\frac {727}{41472}}\, \left| S \right|^{4}
+{\frac {77}{
497664}}\, \left| S \right|^{6}
\end{equation}
\begin{equation}
\label{A.1.11}
M_{L,q,5} =  
{\frac {95}{432}}\, \left| S \right| 
+{\frac {3817}{746496}}\, \left| 
S \right|^{3}
+{\frac {277}{186624}}\,  
 \left| S \right|^{5}+{\frac {23}{1679616}}\,  
 \left| S \right|^{7}
\end{equation}
\begin{equation}
\label{A.1.12}
M_{L,q,6} =  
\frac{621}{5120}\, \left | S \right |^2 +{\frac {1591}{414720}}\,
\left| S \right|^{4}
+{\frac {623}{3732480}}\,  
 \left| S \right|^{6}
+{\frac {1}{995328}}\,  
 \left| S \right|^{8}
\end{equation}
\begin{equation}
\label{A.1.13}
M_{L,q,7} =  
\frac{15}{11264} \left | S \right |^3 +
{\frac {3515}{4866048}}\, \left| S \right|^{5}
+{\frac {
1675}{131383296}}\, \left| S \right|^{7}
+{\frac 
{125}{2364899328}}\, \left| S \right|^{9}
\end{equation}
\begin{equation}
\label{A.16}
M_{G,1} = \sum_{j=0}^2 m_{j,1} \rho^{-j/2} 
\end{equation}
\begin{equation}
\label{A.17}
m_{0,1} = {\frac {784}{3125}} \left (1 +\sqrt{2} \right ) 
+{\frac {5551}{11520}}\, \left| S
 \right|^{2}+{\frac {23}{3732480}}\, \left| S
 \right|^{8}+{\frac {289}{311040}}\,
 \left| S \right|^{6}+{\frac {1417}{207360}}\,
 \left| S \right|^{4}
\end{equation}
\begin{equation}
\label{A.18}
m_{1,1} =  
{\frac {75}{256}}\, \left| S \right|
+
{\frac {2051}{13824}}\, \left| S \right|^{3}+{\frac {
241}{82944}}\, \left| S \right|^{5}+{\frac {5}{
13436928}}\, \left| S \right|^{9}+{\frac {23}{279936
}}\, \left| S \right|^{7}
\end{equation}
\begin{equation}
\label{A.19}
m_{2,1} = 
{\frac {947}{2322432}}\, \left| S \right|^{6}+{
\frac {25}{1504935936}}\, \left| S \right|^{10}+{
\frac {215}{41803776}}\, \left| S \right|^{8}
+{\frac {43}{28672}}\, \left| S \right|^{4}
+{\frac {81}{57344}}\, \left| S \right|^{2}
\end{equation}
\begin{equation}
\label{A.20}
M_{G,2} = \frac{161 |S|^4}{4320} + \frac{53 |S|^2}{160} + \frac{7 |S|^6}{20736}
+ \left ( \frac{995 |S|^3}{6912}  + \frac{11 |S|^7}{373248} + \frac{301 |S|^5}{62208} 
\right ) \rho^{-1/2},  
\end{equation}
\begin{multline}
\label{A.20.1}
M_{G,3} = {\frac {161}{1620}} \left| S \right|^{4}+{\frac {7
}{7776}} \left| S \right|^{6}+{\frac {53}{60}}
\left| S \right|^{2} 
+ \left( {\frac {4975}{13824}}
\left| S \right|^{3}+{\frac {1505}{124416}}
\left| S \right|^{5}+{\frac {55}{746496}}
\left| S \right|^{7} \right) \rho^{-1/2}
\end{multline}
\begin{equation}
\label{8.13.1}
t_5 (\zeta) =
{\frac {161}{1728}}\,{\frac {{S}^{4}}{{\zeta}^{5}}}-{\frac {53}{64}}
\,{\frac {{S}^{2}}{{\zeta}^{3}}}-{\frac {35}{41472}}\,{\frac {{S}^{6}}
{{\zeta}^{7}}} \\, ~
t_6 (\zeta) =
-{\frac {23}{62208}}\,{
\frac {{S}^{7}}{{\zeta}^{8}}}+{\frac {35}{768}}\,{\frac {{S}^{5}}{{
\zeta}^{6}}}-{\frac {1631}{2304}}\,{\frac {{S}^{3}}{{\zeta}^{4}}}
\end{equation}
\begin{equation}
\label{8.13.3}
t_7 (\zeta) =
-{\frac {11}{373248}}
\,{\frac {{S}^{8}}{{\zeta}^{9}}}+{\frac {301}{62208}}\,{\frac {{S}^{6}
}{{\zeta}^{7}}}-{\frac {995}{6912}}\,{\frac {{S}^{4}}{{\zeta}^{5}}}
\end{equation}
\begin{equation}
\label{8.14.1}
u_5 (\zeta) =
{\frac {161}{3456}}\,{\frac {{S}^{4}}{{\zeta}^{3}}}-{\frac {53}{128}}
\,{\frac {{S}^{2}}{\zeta}}-{\frac {35}{82944}}\,{\frac {{S}^{6}}{{
\zeta}^{5}}} \\, ~
u_6 (\zeta) =
{\frac {47}{124416}}\,{\frac {{S}^{7}}{{\zeta}^{6}}}-
{\frac {1631}{41472}}\,{\frac {{S}^{5}}{{\zeta}^{4}}}+{\frac {913}{4608
}}\,{\frac {{S}^{3}}{{\zeta}^{2}}}
\end{equation}
\begin{equation}
\label{8.14.4}
u_7 (\zeta) =
{\frac {173}{746496}}\,{\frac {{
S}^{8}}{{\zeta}^{7}}}-{\frac {3479}{124416}}\,{\frac {{S}^{6}}{{\zeta}^
{5}}}+{\frac {1843}{4608}}\,{\frac {{S}^{4}}{{\zeta}^{3}}} \\, ~
u_8 (\zeta) =
{\frac {11}{559872}}\,{\frac {{S}
^{9}}{{\zeta}^{8}}}-{\frac {301}{93312}}\,{\frac {{S}^{7}}{{\zeta}^{6}
}}+{\frac {995}{10368}}\,{\frac {{S}^{5}}{{\zeta}^{4}}}
\end{equation}
\begin{equation}
\label{8.15.1}
\tau_{5} (\zeta) =  -\frac{161 S^4}{8640 \zeta^5} + \frac{53 S^2}{192 \zeta^3} 
+ \frac{5 S^6}{41472 \zeta^7} \\, ~
\tau_6 (\zeta) = 
\frac{23 S^7}{497664 \zeta^8} - \frac{35 S^5}{4608 \zeta^6}    
+ \frac{1631 S^3}{9216 \zeta^4} 
\end{equation}
\begin{equation}
\label{8.15.3}
\tau_7 (\zeta) = 
\frac{11 S^8}{3359232 \zeta^9} - \frac{43 S^6}{62208 \zeta^7} 
+ \frac{199 S^4}{6912 \zeta^5} 
\end{equation}
\begin{equation}
\label{8.15.4}
{\tilde t}_7 (\zeta) = 
-{\frac {161}{3456}}\,{\frac {{S}^{4}}{{\zeta}^{5}}}
+{\frac {265}{384}}\,
{\frac {{S}^{2}}{{\zeta}^{3}}}
+{\frac {25}{82944}}\,{\frac {{S}^{6}}{{\zeta}^{7}}} \\, ~
{\tilde t}_8 (\zeta) = 
{\frac {23}{165888}}\,{
\frac {{S}^{7}}{{\zeta}^{8}}}-{\frac {35}{1536}}\,{\frac {{S}^{5}}{{
\zeta}^{6}}}+{\frac {1631}{3072}}\,{\frac {{S}^{3}}{{\zeta}^{4}}}
\end{equation}
\begin{equation}
\label{8.15.6}
{\tilde t}_9 (\zeta) = 
{\frac {77}{6718464}}
\,{\frac {{S}^{8}}{{\zeta}^{9}}}-{\frac {301}{124416}}\,{\frac {{S}^{6
}}{{\zeta}^{7}}}+{\frac {1393}{13824}}\,{\frac {{S}^{4}}{{\zeta}^{5}}}
\end{equation}
\begin{equation}
\label{8.16.1}
\nu_{5} (\zeta) = 
-{\frac {161}{10368}}\,{\frac {{S}^{4}}{{\zeta}^{3}}}+{\frac {53}{128}}
\,{\frac {{S}^{2}}{\zeta}}+{\frac {7}{82944}}\,{\frac {{S}^{6}}{{\zeta
}^{5}}} \\,  ~
\nu_{6} (\zeta) =
-{\frac {47}{746496}}\,{\frac {{S}^{7}}{{\zeta}^{6}}}+{\frac {1631}{165888
}}\,{\frac {{S}^{5}}{{\zeta}^{4}}}
-{\frac {913}{9216}}\,{\frac {{S}^{3
}}{{\zeta}^{2}}}
\end{equation}
\begin{equation}
\label{8.16.3}
\nu_7 (\zeta) =
-{\frac {173}{5225472}}\,{\frac 
{{S}^{8}}{{\zeta}^{7}}}+{\frac {3479}{622080}}\,{\frac {{S}^{6}}{{
\zeta}^{5}}}-{\frac {1843}{13824}}\,{\frac {{S}^{4}}{{\zeta}^{3}}} \\, ~
\nu_8 (\zeta) =
-{\frac {11}{4478976}}\,{\frac {{
S}^{9}}{{\zeta}^{8}}}+{\frac {301}{559872}}\,{\frac {{S}^{7}}{{\zeta}^
{6}}}-{\frac {995}{41472}}\,{\frac {{S}^{5}}{{\zeta}^{4}}}
\end{equation}
\begin{equation}
\label{8.16.5}
{\tilde u}_7 (\zeta) = -{\frac {805}{20736}}\,{\frac {{S}^{4}}{{\zeta}^{3}}}
+{\frac {265}{256
}}\,{\frac {{S}^{2}}{\zeta}}
+{\frac {35}{165888}}\,{\frac {{S}^{6}}{{
\zeta}^{5}}} \\, ~ 
{\tilde u}_8 (\zeta) = -{\frac {47}{248832}}\,{\frac {{S}^{7}}{{\zeta}^{6}}}
+{\frac {1631}{55296
}}\,{\frac {{S}^{5}}{{\zeta}^{4}}}-{\frac {913}{3072}}\,{\frac {{S}^{3}
}{{\zeta}^{2}}} 
\end{equation}
\begin{equation}
\label{8.16.8}
{\tilde u}_9 (\zeta) = -{\frac {173}{1492992}}\,{\frac 
{{S}^{8}}{{\zeta}^{7}}}+{\frac {24353}{1244160}}\,{\frac {{S}^{6}}{{
\zeta}^{5}}}-{\frac {12901}{27648}}\,{\frac {{S}^{4}}{{\zeta}^{3}}} \\, ~
{\tilde u}_{10} (\zeta) =  -{\frac {11}{1119744}}\,{\frac {{S
}^{9}}{{\zeta}^{8}}}+{\frac {301}{139968}}\,{\frac {{S}^{7}}{{\zeta}^{
6}}}-{\frac {995}{10368}}\,{\frac {{S}^{5}}{{\zeta}^{4}}}
\end{equation}
\begin{equation}
\label{8.18.1}
p_5 (\zeta) = 
-{\frac {161}{25920}}\,{\frac {{S}^{4}}{{\zeta}^{3}}}+{\frac {53}{192}
}\,{\frac {{S}^{2}}{\zeta}}+{\frac {1}{41472}}\,{\frac {{S}^{6}}{{
\zeta}^{5}}} \\, 
p_6 (\zeta) = 
{\frac {1001}{829440}}\,{\frac {{S}^{5}}{{\zeta}^{4}}}-{\frac {65}{
18432}}\,{\frac {{S}^{3}}{{\zeta}^{2}}}-{\frac {17}{2985984}}\,{\frac 
{{S}^{7}}{{\zeta}^{6}}}
\end{equation}
\begin{equation}
\label{8.18.4}
p_7 (\zeta) =
{\frac {17}{19440}}\,{\frac {{S}^{6}}{{\zeta}^{5}}}-{\frac {185}{
47029248}}\,{\frac {{S}^{8}}{{\zeta}^{7}}}-{\frac {137}{4608}}\,{
\frac {{S}^{4}}{{\zeta}^{3}}} \\, \\
p_8 (\zeta) =
-{\frac {199}{41472}}\,{\frac {{S}^{5}}{{\zeta}^{4}}}+{\frac {43}{
559872}}\,{\frac {{S}^{7}}{{\zeta}^{6}}}-{\frac {11}{40310784}}\,{
\frac {{S}^{9}}{{\zeta}^{8}}}
\end{equation}
\begin{multline}
\label{A.21}
M_{G,4,0}:= {\frac {161}{25920}}\, \left| S \right|^{4}+{\frac {
53}{192}}\, \left| S \right|^{2}+{\frac {1}{41472}}
\, \left| S \right|^{6} 
+\left( {\frac {17}{2985984}
}\, \left| S \right|^{7}+{\frac {1001}{829440}}\,
 \left| S \right|^{5}+{\frac {65}{18432}}\, 
 \left| S \right|^{3} \right) \rho^{-1/2} \\
+ \left( {
\frac {185}{47029248}}\, \left| S \right|^{8}+{
\frac {137}{4608}}\, \left| S \right|^{4}+{\frac {17
}{19440}}\, \left| S \right|^{6} \right) {\rho}^{-1}
+ \left( {\frac {43}{559872}}\, \left| S \right|^{7}
+{\frac {11}{40310784}}\, \left| S \right|^{9}+{
\frac {199}{41472}}\, \left| S \right|^{5} \right) \rho^{-3/2}
\end{multline}
\begin{multline}
\label{A.22}
M_{G,4,1} = 
{\frac {|S|^6}{6912}}\, + {\frac {161 |S|^4
}{6480}}\, +{\frac {53 |S|^2}{96}}\,
+ \left( {\frac {2401 |S|^5}{92160}}
+{\frac {761 |S|^3}{2048}}\,
+{\frac {497 |S|^7}{2985984}}\,
\right) \rho^{-1/2} \\
 +\left( {\frac {1711 |S|^6}{155520}}\, 
+{\frac {3083 |S|^8}{47029248}}\, +{
\frac {3673 |S|^4}{13824}}\, \right) 
\rho^{-1}+ \left( {\frac {199 |S|^5}{4608}}\,  
+{\frac {559 |S|^7}{559872}}\,  
+{\frac {187 |S|^9}{40310784}}
\right) \rho^{-3/2} 
\end{multline}
\begin{multline}
\label{A.23}
M_{G,5} =
{\frac {1417}{331776}}\, \left| S \right|^{4}+{
\frac {23}{5971968}}\, \left| S \right|^{8}+{\frac {
289}{497664}}\, \left| S \right|^{6}+{\frac {5551}{
18432}}\, \left| S \right|^{2}+{\frac {196}{625}} \\
+
 \left( {\frac {241}{138240}}\, \left| S \right|^{5}
+{\frac {2051}{23040}}\, \left| S \right|^{3}+{
\frac {45}{256}}\, \left| S \right| +{\frac {1}{4478976}}\,  
 \left| S \right|^{9}+{\frac {23}{466560}}\, \left| 
S \right|^{7} \right) \rho^{-1/2} \\
+ \left( {\frac {25
}{2579890176}}\, \left| S \right|^{10}+{\frac {215}{
71663616}}\, \left| S \right|^{8}+{\frac {947}{
3981312}}\, \left| S \right|^{6}+{\frac {43}{49152}}
\, \left| S \right|^{4}+{\frac {27}{32768}}\,
 \left| S \right|^{2} \right) {\rho}^{-1}
\end{multline}
\begin{equation}
\label{A.24}
M_{G,6} = {\frac {161}{3456}} \left| S \right|^{4}+{\frac {
35}{82944}} \left| S \right|^{6}+{\frac {53}{128}}
\left| S \right|^{2} 
+ \left( {\frac {199}{1152}}
\left| S \right|^{3}+{\frac {301}{51840}}
\left| S \right|^{5}+{\frac {11}{311040}}
\left| S \right|^{7} \right) \rho^{-1/2}
\end{equation}
\begin{multline}
\label{A.25}
M_{G,7} = 
{\frac {1771}{57600}} \left| S \right|^{4}+{\frac 
{11}{55296}}\left| S \right|^{6}+{\frac {583}{
1280}}\left| S \right|^{2} 
+ \left( {\frac {37513}
{138240}} \left| S \right|^{3}+{\frac {161}{13824}
} \left| S \right|^{5}+{\frac {529}{7464960}}
 \left| S \right|^{7} \right) \rho^{-1/2}
\\
+
\left( {\frac {12139}{290304}}\, \left| S \right|^{4}
+{\frac {2623}{2612736}}\, \left| S \right|^{6}
+{
\frac {671}{141087744}} \left| S \right|^{8}
 \right) \rho^{-1} 
\end{multline}

\subsection{Values of intermediate constants for $\rho=3$} The numerical values of these constants might be helpful to the reader who would like to double-check  the estimates. These are: 
{\footnotesize \begin{eqnarray*}
  J_M = 0.282580 {\scriptscriptstyle {\scriptscriptstyle \cdots}} , j_m = 0.64374{\scriptscriptstyle \cdots} ,  Y_{1,M} = 1.16314{\scriptscriptstyle \cdots} , Y_{1,R,M} = 0.132618{\scriptscriptstyle \cdots}   
, 
E_M = 0.0490292{\scriptscriptstyle \cdots}\\
z_{2,R,M} = 0.54226{\scriptscriptstyle \cdots} , z_{2, M} = 0.91863 {\scriptscriptstyle \cdots} , M_{q} = 0.066702{\scriptscriptstyle \cdots} , M_{L,q} = 
0.075708{\scriptscriptstyle \cdots} , V_M = 0.2239{\scriptscriptstyle \cdots} , T_M =0.0385{\scriptscriptstyle \cdots} 
\\
M_1 = 1.13838{\scriptscriptstyle \cdots} , 
M_2 = 0.04303{\scriptscriptstyle \cdots} ,
M_3 = 0.28346{\scriptscriptstyle \cdots} ,
M_4 = 0.45227{\scriptscriptstyle \cdots} , 
M_5 = 0.05430{\scriptscriptstyle \cdots} ,
M_6 = 0.00231{\scriptscriptstyle \cdots} ,
M_7 = 0.02018{\scriptscriptstyle \cdots}
\end{eqnarray*}
}
\section{Acknowledgments} The work of O.C. and S.T  was partially supported by the NSF
grant DMS 1108794.

\end{document}